\documentclass[10pt]{amsart}
\usepackage{amstext,amsfonts,amssymb,amscd,amsbsy,amsmath}
\usepackage{enumerate}
\usepackage[all]{xy}
\SelectTips{eu}{}

\newcommand{\bbZ}{\ensuremath{\mathbb{Z}}}
\newcommand{\bbQ}{\ensuremath{\mathbb{Q}}}
\newcommand{\bbN}{\ensuremath{\mathbb{N}}}
\newcommand{\bx}{\ensuremath{\boldsymbol{x}}}
\newcommand{\by}{\ensuremath{\boldsymbol{y}}}
\newcommand{\fa}{\ensuremath{\mathfrak a}}
\newcommand{\fb}{\ensuremath{\mathfrak b}}
\newcommand{\fc}{\ensuremath{\mathfrak c}}

\newcommand{\fp}{\ensuremath{\mathfrak p}}
\newcommand{\fq}{\ensuremath{\mathfrak q}}
\newcommand{\fr}{\ensuremath{\mathfrak r}}
\newcommand{\fs}{\ensuremath{\mathfrak s}}
\newcommand{\ft}{\ensuremath{\mathfrak t}}

\newcommand{\ta}{\ensuremath{\mathsf T}}
\newcommand{\eR}{\ensuremath{\varepsilon_R}}
\newcommand{\eS}{\ensuremath{\varepsilon_S}}

\newcommand{\iR}{\ensuremath{\iota_R}}
\newcommand{\iS}{\ensuremath{\iota_S}}

\DeclareMathOperator{\HH}{H}
\DeclareMathOperator{\Ker}{Ker}
\DeclareMathOperator{\Ima}{Im}
\DeclareMathOperator{\Ext}{Ext}
\DeclareMathOperator{\Tor}{Tor}
\DeclareMathOperator{\Hom}{Hom}
\DeclareMathOperator{\rank}{rank}
\DeclareMathOperator{\soc}{Soc}
\DeclareMathOperator{\depth}{depth}
\DeclareMathOperator{\edim}{edim}
\DeclareMathOperator{\gcl}{gcl}
\DeclareMathOperator{\codim}{codim}
\DeclareMathOperator{\type}{type}
\DeclareMathOperator{\length}{\ell}

\newcommand{\lra}{\ensuremath{\longrightarrow}}

\newcommand{\xra}{\ensuremath{\xrightarrow}}
\newcommand{\xla}{\ensuremath{\xleftarrow}}

\newcommand{\poinc}[2]{\ensuremath{P^{#1}_{#2}}}

\newcommand{\wt}{\widetilde}
\newcommand{\ov}{\overline}

\theoremstyle{plain}
\newtheorem{theorem}{Theorem}[section]
\newtheorem{proposition}[theorem]{Proposition}
\newtheorem{corollary}[theorem]{Corollary}
\newtheorem{lemma}[theorem]{Lemma}

\theoremstyle{definition}

\newtheorem{subsec}[theorem]{}
\newtheorem{example}[theorem]{Example}

\newtheorem{ssubsec}{}
\numberwithin{ssubsec}{theorem}

\theoremstyle{remark}
\newtheorem{remark}[theorem]{Remark}

\newtheorem*{Remark}{Remark}

\numberwithin{equation}{theorem}

\begin{document}

\title[Connected sums of rings]
      {Connected sums of Gorenstein local rings}

\author[H.~ Ananthnarayan]{H.~ Ananthnarayan}
\address{Department of Mathematics,
   University of Nebraska, Lincoln, NE 68588, U.S.A.}
     \email{ahariharan2@math.unl.edu}

\author[L.~L.~Avramov]{Luchezar L.~Avramov}
\address{Department of Mathematics,
   University of Nebraska, Lincoln, NE 68588, U.S.A.}
     \email{avramov@math.unl.edu}

\author[W.~F.~Moore]{W.~Frank Moore}
\address{Department of Mathematics,
   Cornell University, Ithaca, NY 14853, U.S.A.}
     \email{frankmoore@math.cornell.edu}

\date{\today}

\thanks{L.L.A.~was partly supported by NSF grant DMS 0803082.}

\subjclass[2000]{Primary 13D07, 13D40}

\dedicatory{To Gerson Levin, on his seventieth birthday.}

\keywords{Local ring, fiber product, connected sum, Gorenstein ring, 
Gorenstein colength, Poincar\'e series, Golod homomorphism, 
cohomology algebra, homotopy Lie algebra.}

  \begin{abstract}
Given surjective homomorphisms $R\to T\gets S$ of local rings, and
ideals in $R$ and $S$ that are isomorphic to some $T$-module $V$,
the \emph{connected sum} $R\#_TS$ is defined to be ring obtained by
factoring out the diagonal image of $V$ in the fiber product $R\times_TS$.
When $T$ is Cohen-Macaulay of dimension $d$ and $V$ is a canonical
module of $T$, it is proved that if $R$ and $S$ are Gorenstein of dimension
$d$, then so is $R\#_TS$.  This result is used to study how closely an
artinian ring can be approximated by a Gorenstein ring mapping onto~it.
When $T$ is regular, it is shown that $R\#_TS$ almost never is a complete
intersection ring.  The proof uses a presentation of the cohomology algebra
$\Ext^*_{R\#_kS}(k,k)$ as an amalgam of the algebras $\Ext^*_{R}(k,k)$
and $\Ext^*_{S}(k,k)$ over isomorphic polynomial subalgebras generated
by one element of degree $2$.
  \end{abstract}

\maketitle

\section*{Introduction}

We introduce, study, and apply a new construction of local Gorenstein
rings.

The starting point is the classical fiber product $R\times_TS$ of a
pair of surjective homomorphisms $\eR\colon R\to T\gets S\ {:}\,\eS$ of
local rings.  It is well known that this ring is local, but until recently, little 
was known about its properties.  In Proposition \ref{cmProd} we show that
if $R$, $S$, and $T$ are Cohen-Macaulay of dimension $d$, then so is
$R\times_TS$, but this ring is Gorenstein only in trivial cases.  When $\eR=\eS$, 
D'Anna \cite{D} and Shapiro \cite{Sh} proposed and partly proved a criterion for 
$R\times_TR$ to be Gorenstein.  We complete and strengthen their results in 
Theorem \ref{thm:danna}: $R\times_TR$ Is Gorenstein if and only if 
$R$ is Cohen-Macaulay and $\Ker\eR$ is a canonical module for $R$. 

Our main construction involves, in addition to the ring homomorphisms
$\eR$ and $\eS$, a $T$-module $V$ and homomorphisms $\iR\colon V\to R$ of $R$-modules
and $\iS\colon V\to S$ of $S$-modules, for the structures induced through
$\eR$ and $\eS$, respectively.  When these maps satisfy $\eR\iR=\eS\iS$,
we define a \emph{connected sum} ring by the formula
 \[
R\#_TS=(R\times_TS)/\{(\iR(v),\iS(v))\mid v\in V\}\,.
  \]

In case $R$, $S$, and $T$ have dimension $d$ (for some $d\ge0$), $R$
and $S$ are Gorenstein, $T$ is Cohen-Macaulay, and $V$ is a canonical
module for $T$, one can choose $\iR$ and $\iS$ to be isomorphisms
onto $(0:\Ker(\eR))$ and $(0:\Ker(\eS))$, respectively.  In Theorem
\ref{gorenstein} we prove that if $\eR\iR=\eS\iS$ holds, then $R\#_TS$
is Gorenstein of dimension $d$.

Much of the paper is concerned with Gorenstein rings of this form. 

As a first application, we study how efficiently an artinian local ring
can be approximated by a Gorenstein artinian ring mapping onto it. One
numerical measure of proximity is given by the Gorenstein colength of
an artinian ring, introduced in \cite{A1}. We obtain new estimates for
this invariant. We use them in the proof of Theorem \ref{gcl1} to remove
a restrictive hypothesis from a result of Huneke and Vraciu \cite{HV},
describing homomorphic images of Gorenstein local rings modulo their
socles.

When $d = 0$ and $T$ is a field, the construction of $R\#_T S$ mimics
the expression for the cohomology algebra of a connected sum $M\#N$ of
compact smooth manifolds $M$ and $N$, in terms of the cohomology algebras
of $M$ and $N$; see Example \ref{manifolds}. This analogy provides the
name and the notation for connected sums of rings.

The topological analogy also suggests that connected sums may be used
for classifying Cohen-Macaulay quotient rings of Gorenstein rings.
The corresponding classification problem is, in a heuristic sense, dual to
the one approached through Gorenstein linkage: Whereas linkage operates
on the set of \emph{Cohen-Macaulay quotients of a fixed Gorenstein ring}
$R$, connected sums operate on the set of \emph{Gorenstein rings with
a fixed Cohen-Macaulay quotient ring} $T$.

This point of view raises the question of identifying those rings $Q$
that are \emph{indecomposable}, in the sense that an isomorphism $Q\cong
R\#_TS$ implies $Q\cong R$ or $Q\cong S$.  In Theorem \ref{regular}
we show that if $T$ is regular and $Q$ is complete intersection,
then either $Q$ is indecomposable, or it is a connected sum of two
quadratic hypersurface rings.  The argument uses the structure of
the algebra $\Ext^*_{R\#_TS}(T,T)$, when $R$ and $S$ are artinian
and $T$ is a field.  In Theorem \ref{connected} we show that it is an
amalgam of $\Ext^*_{R}(T,T)$ and $\Ext^*_{S}(T,T)$ over a polynomial
$T$-subalgebra, generated by an element of degree $2$.  The machinery
for the proof is fine-tuned in Sections \ref{sec:Cohomology algebras}
and \ref{sec:Cohomology of fiber products}.

\section{Fiber products}
  \label{sec:Pdcts}

The \emph{fiber product} of a diagram of homomorphisms of commutative rings
 \begin{equation}
   \label{diagramProd}
  \begin{gathered}
\xymatrixrowsep{1pc}
\xymatrixcolsep{1pc}
\xymatrix{
R
\ar[dr]^{\eR}
\\
&{\quad}T{\quad}
  \\
S
\ar[ur]_{\eS}
}
  \end{gathered}
   \end{equation}
is the subring of $R\times S$, defined by the formula
  \begin{equation}
     \label{eq:Prod}
R\times_TS=\{(x,y)\in R\times S\mid \eR(x)=\eS(y)\}\,.
  \end{equation}

If $R\xla{\alpha_R}A\xra{\alpha_S}S$ are surjective homomorphisms 
of rings, then for $T=R\otimes_AS$, $\eR(r)=r\otimes1$, and $\eS(s)=1\otimes s$
the map $a\mapsto(\alpha_R(a),\alpha_S(a))$ is a
surjective homomorphism of rings $A\to R\times_TS$ with kernel
$\Ker(\alpha_R)\cap\Ker(\alpha_S)$, whence
   \begin{equation}
     \label{eq:presentation}
R\times_TS\cong A/(\Ker(\alpha_R)\cap\Ker(\alpha_S))\,.
  \end{equation}

\medskip

In the sequel, the phrase \emph{$(Q,\fq,k)$ is a local ring}
means that $Q$ is a commutative noetherian ring with unique 
maximal ideal $\fq$ and residue field $k=Q/\fq$. 

\medskip

\emph{The following setup and notation are in force for the rest of this section:}

  \begin{subsec}
    \label{setupProd}
The rings in diagram \eqref{diagramProd} are local: 
$(R,\fr,k)$, $(S,\fs,k)$, and $(T,\ft,k)$.

The maps $\eR$ and $\eS$ are surjective; set $I=\Ker(\eR)$ 
and $J=\Ker(\eS)$, and also
  \[
P=R\times_TS\,.
  \]

Let $\eta$ denote the inclusion of rings $P\to R\times S$, and let 
$R\gets R\times S\to S$ be the canonical maps. Each (finite) module 
over $R$ or $S$ acquires a canonical structure of (finite) $P$-module 
through the composed homomorphisms of rings $R\gets P\to S$
(finiteness is preserved because these maps are surjective). 

The rings and ideals above are related through exact sequences
of $P$-modules
  \begin{gather}
    \label{eq:fiber1}
0\lra I\oplus J\lra R\oplus S\xra{\,\eR\oplus\eS\,} T\oplus T\lra0
  \\
    \label{eq:fiber2}
0\lra R\times_TS\xra{\,\eta\,} R\oplus S\xra{\,(\eR,-\eS)\,} T\lra 0  
  \end{gather}

A length count in the second sequence yields the relation
  \begin{equation}
    \label{eq:lengthProd}
\length(R\times_TS)+\length(T) =\length(R)+\length(S)\,.
  \end{equation}
      \end{subsec}

For completeness, we include a proof of the following result; see
\cite[19.3.2.1]{Gr}.

  \begin{lemma}
 \label{local1}
The ring $R\times_TS$ is local, with maximal ideal $\fp=\fr\times_\ft\fs$.
  \end{lemma}

  \begin{proof} 
The rings $R$ and $S$ are quotients of $P$, so they are noetherian 
$P$-modules.  Thus, the $P$-module $R\oplus S$ is noetherian, and 
hence so is its submodule $P$.

If $(r,s)$ is in $P$, but not in $\fr\times_\ft\fs$, then $r$ is not in $\fr$, 
so $r$ is invertible in $R$.  Since $\eS$ is surjective, there exists $s'\in S$ 
with $\eS(s')=\eR(r^{-1})$. One then has $\eS(s's)=\eR(r^{-1})\eR(r)=1$, 
so $a=s's$ is an invertible element of $S$.  Now $(r^{-1},a^{-1}s')$
is in $P$, and it  satisfies $(r^{-1},a^{-1}s')(r,s)=(r^{-1}r,a^{-1}s's)=(1,1)$.
  \end{proof}

For any sequence $\bx$ of elements of $P$ and $P$-module $M$, we let 
$\HH_n(\bx,M)$ denote the $n$th homology module of the Koszul complex 
on $\bx$ with coefficients in $M$.  

  \begin{lemma}
 \label{local2}
When $\bx$ is a $T$-regular sequence in $R\times_TS$ and $\ov M$ 
denotes $M/\bx M$ for each $(R\times_TS)$-module $M$, there 
is an isomorphism of rings
  \[
\ov{R\times_TS}\cong\ov R\times_{\ov T}\ov S
  \]
and there are exact sequences of $(\ov{R\times_TS})$-modules
  \begin{gather}
    \label{eq:fiberOv1}
0\lra\ov I\oplus\ov J\lra\ov R\oplus\ov S\xra{\,\ov\eR\oplus\ov\eS\,}
\ov T\oplus\ov T\lra0
  \\
    \label{eq:fiberOv2}
0\lra\ov{R\times_TS}\xra{\ \ov\eta\ }\ov R\oplus\ov
S\xra{\,(\ov\eR,-\ov\eS)\,}\ov T\lra 0
  \qedhere \end{gather}

The sequence $\bx$ is $R\times_TS$-regular if and only if it is $R$-regular and 
$S$-regular.
  \end{lemma}

  \begin{proof}
One has $\HH_n(\bx,T)=0$ for $n\ge1$, so \eqref{eq:fiber1} induces an 
exact sequence of Koszul homology modules, which contains \eqref{eq:fiberOv1}.
It also gives an isomorphism 
  \begin{align*}
\HH_1(\bx,P)&\cong\HH_1(\bx,R)\oplus\HH_1(\bx,S)\,,
  \end{align*}
which shows that $\bx$ is $P$-regular if and only if it is $R$-regular and 
$S$-regular.  

The exact sequence of Koszul homology modules induced by \eqref{eq:fiber2}
contains the exact sequence \eqref{eq:fiberOv2}, which, in turn 
implies the desired isomorphism of rings.
  \end{proof}

We relate numerical invariants of $P$ to the corresponding 
ones of $R$, $S$, and~$T$.

  \begin{subsec}
    \label{invariants}
When $Q$ is a local ring and $N$ a finite $Q$-module, $\dim_QN$ denotes
its Krull dimension and $\depth_QN$ its depth of $N$.  Recall that if $P\to Q$ 
is a finite homomorphism of local rings, then one has
$\dim_PN=\dim_QN$ and $\depth_PN=\depth_QN$.

We set $\dim Q=\dim_QN$ and $\depth Q=\depth_QQ$;
thus, there are equalities $\dim Q=\dim_PQ$ and $\depth Q=\depth_PQ$.  

Recall that $\edim Q$ denotes the \emph{embedding dimension} of $Q$, defined 
to be the minimal number of generators of its maximal ideal.
  \end{subsec}
  
 \begin{lemma}
 \label{local3}
The following (in)equalities hold:
   \begin{align}
     \label{eq:local3.4}
\edim(R\times_TS)&\ge\edim R+\edim S-\edim T\,.  
  \\
     \label{eq:local3.1}
\dim(R\times_TS)&=\max\{\dim R\,,\dim S\}
\ge\min\{\dim R\,,\dim S\}\ge\dim T\,.
  \\
     \label{eq:local3.2}
\depth(R\times_TS)&\ge\min\{\depth R\,,\depth S\,,\,\depth T+1\}\,.
  \\
     \label{eq:local3.3}
\depth T&\ge\min\{\depth R, \depth S, \depth(R \times_T S) -1\}\,.
  \end{align}
    \end{lemma}

  \begin{proof}
Lemma \ref{local1} gives an exact sequence of $P$-modules
  \[
0\to\fp\to\fr\oplus\fs\to\ft\to 0
  \]
Tensoring it with $P/\fp$ over $P$, we get an exact sequence 
of $k$-vector spaces 
  \[
\fp/\fp^2\to\fr/\fr^2\oplus\fs/\fs^2\to\ft/\ft^2\to 0
  \]
because we have $\fp\fr=\fr^2$, $\fp\fs=\fs^2$, and $\fp\ft=\ft^2$, due to
the surjective homomorphisms $R\gets P\to S\to T\gets R$.  These
maps also give $\min\{\dim R\,,\dim S\}\ge\dim T$ and 
$\dim P\ge\max\{\dim R\,,\dim S\}$, while the inclusion $\eta$ from
\eqref{eq:fiber2} yields 
  \[
\max\{\dim_PR\,,\dim_PS\}=\dim_P(R\oplus S)\ge\dim_P P\,.
  \]
For \eqref{eq:local3.2} and \eqref{eq:local3.3}, apply the 
Depth Lemma, see \cite[1.2.9]{BH}, to \eqref{eq:fiber2}.  
  \end{proof}

For a local ring $(Q,\fq,k)$ and $Q$-module $N$, set $\soc N=\{n\in N\mid \fq n=0\}$. 
When $\bx$ is a maximal $N$-regular sequence, $\rank_k\soc(N/\bx N)$ 
is a positive integer that does not depend on $\bx$, see \cite[1.2.19]{BH}, denoted $\type_QN$.
Set $\type Q=\type_QQ$; thus, $Q$ is Gorenstein if and only if it is Cohen-Macaulay and
$\type Q=1$. 

We interpolate a useful general observation that uses fiber producs.

  \begin{lemma}
    \label{lem:socle}
Let $(Q,\fq,k)$ be a local ring and $W$ a $k$-subspace of $(\soc(Q)+\fq^2)/\fq^2$.

There exists a ring isomorphism $Q\cong B\times_kC$, where $(B,\fb,k)$ 
and $(C,\fc,k)$ are local rings, such that $\fc^2=0$ and $\fc\cong W$. 

If $W=\soc(Q)+\fq^2)/\fq^2$, then $\soc(B)\subseteq\fb^2$.
  \end{lemma}

  \begin{proof}
When $\soc(Q)$ is in $\fq^2$, set $B=Q$ and $C=k$.  Else, pick in $\soc Q$ a set 
$\bx$ that maps bijectively to a basis of $W$, then choose in 
$\fq$ a set $\by\subset\fq$, so that $\bx\cup\by$ maps bijectively to a basis of 
$\fq/\fq^2$.  Set $B=Q/(\bx)$ and $C=Q/(\by)$.  One then has $\fq=(\bx)+(\by)$, 
hence $B\otimes_QC\cong k$, and also $(\bx)\cap(\by)=0$, so $Q\cong B\times_k
C$ by \eqref{eq:presentation}.  The desired properties of $B$ and $C$ are verified 
by elementary calculations.
  \end{proof}

The next two results concern ring-theoretic properties of fiber products.

 \begin{proposition}
 \label{cmProd}
Assume that $T$ is Cohen-Macaulay, and set $d=\dim T$.  

The ring $R\times_TS$ is Cohen-Macaulay of dimension $d$ if and only 
if $R$ and $S$ are.

When $R\times_TS$ is Cohen-Macaulay of dimension $d$ the following 
inequalities hold:
  \begin{align*}
\type R+\type S
&\ge\type (R\times_TS)
  \\
&\ge\max\{\type R+\type S-\type T,\type_RI+\type_SJ\} \,.
  \end{align*}
If, in addition, $I$ and $J$ are non-zero, then $R\times_TS$ is not
Gorenstein.
  \end{proposition}

  \begin{proof}
The first assertion follows directly from Lemmas \ref{local2} an \ref{local3},
so assume that $P$ is Cohen-Macaulay of dimension $d$.  Choosing in $P$
an $(P\oplus T)$-regular sequence of length~$d$, from \eqref{eq:fiberOv2} we
get an exact sequence of $k$-vector spaces
  \[
0\lra\soc(\ov{P})\xra{\,\soc\ov\eta\,}\soc\ov R\oplus\soc\ov S
\xra{\,(\soc\ov\eR,-\soc\ov\eS)\,}\soc\ov T
  \]
It provides the inequalities involving $\type R$ and $\type S$.  Formula 
\eqref{eq:fiberOv1} gives $\ov\eR(\soc \ov I)=0= \ov\eS(\soc \ov
J)$, so the sequence above yields $\ov\eta(\soc\ov P)\supseteq\soc\ov
I\oplus\soc\ov J$.  When $I\ne0\ne J$ holds, we get $\ov I\ne0\ne\ov J$
by Nakayama's Lemma.  Since $\ov R$ and $\ov S$ are artinian, one has
$\soc\ov I\ne0\ne\soc\ov J$, whence $\type P\ge 2$.
  \end{proof}

When $\eR\colon R\to R/I$ is the canonical map and $\eS=\eR$, the ring 
$R \bowtie I=R \times_{R/I}R$ has been studied 
under the name \emph{amalgamated duplication of $R$ along $I$}.   
We complete and strengthen results of D'Anna and Shapiro:

  \begin{theorem}
    \label{thm:danna}
Let $R$ be a local ring, $d$ its Krull dimension, and $I$ a non-unit ideal.

The ring $R \bowtie I$ is Cohen-Macaulay if and only if $R$ is 
Cohen-Macaulay and $I$ is a maximal Cohen-Macaulay $R$-module.

The ring $R \bowtie I$ is Gorenstein if and only if $R$ is 
Cohen-Macaulay and $I$ is a canonical module for~$R$, and then
$R/I$ is Cohen-Macaulay with $\dim(R/I)=d-1$.
  \end{theorem}

We start by listing those assertions in the theorem that are already known.

  \begin{subsec}
Assume that the ring $R$ is Cohen-Macaulay.
  \begin{ssubsec}
    \label{parts1}
If $I$ is a maximal Cohen-Macaulay module, then $R \bowtie I$ is 
Cohen-Macaulay: This is proved by D'Anna in \cite[Discussion 10]{D}.
  \end{ssubsec}
  \begin{ssubsec}
    \label{parts2}
If $I$ is a canonical module for $R$, then $R \bowtie I$ is Gorenstein:
This follows from a result of Eisenbud; see \cite[Theorem 12]{D}.
 \end{ssubsec}
  \begin{ssubsec}
    \label{parts3}
If $R \bowtie I$ is Gorenstein \emph{and $I$ contains a regular element}, 
then $I$ is a canonical module for $R$: In D'Anna's proof of \cite[Theorem 11]{D},
this is deduced from \cite[Proposition 3]{D}; the italicized part of the hypothesis does 
not appear in the statement of that proposition, but Shapiro 
\cite[2.1]{Sh} shows that it is needed.
 \end{ssubsec}
  \begin{ssubsec}
    \label{parts4}
If $R \bowtie I$ is Gorenstein and $\dim R=1$, then $I$ contains a regular 
element:  This is proved by Shapiro, see \cite[2.4]{Sh}; in the statement  of that 
result it is also assumed that $R$ reduced, but this hypothesis is not used in the proof.
 \end{ssubsec}
  \end{subsec}

\begin{proof}[Proof of Theorem \emph{\ref{thm:danna}}]
Set $P = R \bowtie I$ and $d=\dim R$; thus, $\dim P=d$ by \eqref{eq:local3.1}.

We obtain the first assertion from a slight variation of the argument for~\ref{parts1}.
The map $R\to R\times R$, given by $r\mapsto(r,r)$, defines a homomorphisms of 
rings $R\to P$ that turns $P$ into a finite $R$-module.  Thus, $P$ is a Cohen-Macaulay 
ring if and only if it is Cohen-Macaulay as an $R$-module; see \ref{invariants}.  This 
module is isomorphic to $R\oplus I$, because each element $(r,s)\in P$ has a unique 
expression of the form $(r,r)+(0,s-r)$.  It follows that $P$ is Cohen-Macaulay if and 
only if $R$ is Cohen-Macaulay and $I$ is a maximal Cohen-Macaulay $R$-module.

In view of \ref{parts2}, for the rest of the proof we may assume $P$ Gorenstein.

Set $T= R/I$.  We have $\depth T \geq d - 1\ge0$ by \eqref{eq:local3.3} and 
Proposition~\ref{cmProd}. By the already proved assertion, $R$ is 
Cohen-Macaulay with $\depth R=d$, so we can choose in $P$ a $T$-regular and $R$-regular
sequence  $\bx$ of length $d-1 $; for each $P$-module $M$ set $\ov M=M/\bx M$. 
By Lemma~\ref{local2}, $\bx$ is $P$-regular, $\ov I$ is an ideal in $\ov R$
and there are isomorphisms of rings $\ov T\cong \ov R/\ov I$ and 
$\ov{P}\cong \ov R \bowtie \ov I$.  As $\ov R$ is Cohen-Macaulay with $\dim \ov R=1$
and $\ov P$ is Gorenstein, \ref{parts4} shows that $\ov I$ contains an 
$\ov R$-regular element.  This yield $\dim \ov T = 0$, hence  $\dim T = d - 1$, so $T$ 
is Cohen-Macaulay.  Since $R$ is Cohen-Macaulay as well, we have 
$\operatorname{grade}_RT=\dim R-\dim T=1$, so $I$ contains a regular element, 
and hence $I$ is a canonical module for $R$, due to \ref{parts3}.
  \end{proof}

\section{Connected sums}
  \label{sec:ConnSum}

\setcounter{equation}{0}

A \emph{connected sum diagram} of commutative rings is a commutative 
diagram
  \begin{equation}
    \label{diagramSum}
  \begin{gathered}
\xymatrixrowsep{1pc}
\xymatrixcolsep{2pc}
 \xymatrix{
& R \ar[dr]^{\eR}
 \\
V \ar[ur]^{\iR} \ar[dr]_{\iS}
&& T
 \\
& S \ar[ur]_{\eS}
}
   \end{gathered}
  \end{equation}
where $V$ is a $T$-module, $\iR$ a homomorphism of $R$-modules
(with $R$ acting on $V$ through $\eR$) and $\iS$ a homomorphism 
of $S$-modules (with $S$ acting on $V$ via $\eS$).  

Evidently, $\{(\iR(v),\iS(v))\in R\times S\mid v\in V\}$ is an ideal of 
$R\times_TS$.  We define the \emph{connected sum of $R$ and $S$
along the diagram \eqref{diagramSum}} to be the ring
  \begin{equation}
    \label{eq:Sum}
R\#_TS=(R\times_TS)/\{(\iR(v),\iS(v))\mid v\in V\}\,.
  \end{equation}
As in the case of fiber products, the maps in the diagram are suppressed
from the notation, although the resulting ring does depend on them; see 
Example \ref{fermat}.  The choices of name and notation are explained 
in Example \ref{manifolds}.  

\medskip

\emph{We fix the setup and notation for this section as follows:}

  \begin{subsec}
 \label{setupSum}
The rings in diagram \eqref{diagramSum} are local: $(R,\fr,k)$, $(S,\fs,k)$ 
and $(T,\ft,k)$.  

The maps $\eR$ and $\eS$ are surjective; set $I=\Ker(\eR)$, $J=\Ker(\eS)$, 
also
  \[
P=R\times_TS \quad\text{and}\quad Q=R\#_TS\,.
  \]

The maps $\iR$ and $\iS$ are injective, so there are exact sequences 
of finite $P$-modules 
  \begin{gather}
    \label{eq:injection1}
0\lra V\oplus V\xra{\,{\iR}\oplus{\iS}\,} R\oplus S\lra R/\iR(V)\oplus
S/\iS(V)\lra0
  \\
    \label{eq:injection2}
0\lra V\xra{\,{\iota}\,} R\times_TS\xra{\,{\kappa}\,}R\#_TS\lra 0
 \end{gather}
where $\iota\colon v\mapsto(\iR(v),\iS(v))$ and $\kappa$ is the canonical surjection.

A length count in \eqref{eq:injection2}, using formula \eqref{eq:lengthProd}, yields
  \begin{equation}
    \label{eq:lengthSum}
\length(R\#_TS)+\length(T)+\length(V)=\length(R)+\length(S) \,.
 \end{equation}
  \end{subsec}

  \begin{subsec}
 \label{trivial}
The ring $Q$ is local and we write $(Q,\fq,k)$, unless $Q=0$.  The condition 
$Q=0$ is equivalent to $\iR(V)=R$, and also to $\iS(V)=S$:  This follows
from the fact that $(P,\fp,k)$ is a local ring with $\fp=\fr\times_T\fs$, see 
Lemma \ref{local1}.

When $I=0$ one has $R\times_TS\cong S$, hence $R\#_TS\cong S/\iS(V)$.   
  \end{subsec}

    \begin{lemma}
 \label{regularSum}
If a sequence $\bx$ in $R\times_TS$ is regular on 
$R/\iR(V)$, $S/\iS(V)$, $T$, and~$V$, then it is also
regular on $R$, $S$, $R\times_TS$, and $R\#_TS$, 
and there is an isomorphism
  \[
\ov{R\#_TS}\cong \ov R\#_{\ov T}\ov S
  \]
of rings, where $\ov M$ denotes $M/\bx M$ for every $R\times_TS$-module $M$.
  \end{lemma}

  \begin{proof}
The sequence \eqref{eq:injection1} induces an exact sequence of Koszul
homology modules
  \begin{equation}
    \label{eq:injection}
0\lra\HH_1(\bx, R)\oplus\HH_1(\bx, S)\lra0 \lra\ov V\oplus \ov
V\xra{\,\ov{\iR}\oplus\ov{\iS}\,}\ov R\oplus \ov S
  \end{equation}
It follows that $\bx$ is $R$-regular and $S$-regular.  Lemma \ref{local2}
shows that it is also $P$-regular, so \eqref{eq:injection2} induces an
exact sequence of Koszul homology modules
  \[
0\lra\HH_1(\bx,Q)\lra\ov V\xra{\ {\ov\iota}\ }\ov P\xra{\ \ov\kappa\ }
\ov Q\lra 0
  \]
Note that $\ov\iota$ is  equal to the composition of the diagonal
map $\ov V\to\ov V\oplus\ov V$ and $\ov{\iR}\oplus\ov{\iS}\colon \ov
V\oplus\ov V\to\ov R\oplus \ov S$.  Both are injective, the second one
by \eqref{eq:injection}, so $\ov\iota$ is injective as well.   We get
$\HH_1(\bx,Q)=0$, so $\bx$ is $Q$-regular.  After identifying $\ov P$
and $\ov R\times_{\ov T}\ov S$ through Lemma \ref{local2}, we get $\ov
Q\cong\ov R\#_{\ov T}\ov S$ from the injectivity of $\ov\iota$.
  \end{proof}

  \begin{proposition}
 \label{cmSum}
If the rings, $R/\iR(V)$, $S/\iS(V)$, $T$, and the $T$-module
$V$ are Cohen-Macaulay of dimension $d$, then so are the
rings $R$, $S$, $R\times_TS$, and $R\#_TS$. 
  \end{proposition}

  \begin{proof}
The exact sequence \eqref{eq:injection1} implies that $R$ and $S$ are
Cohen-Macaulay of dimension $d$.  Proposition \ref{cmProd} then shows
that so is $P$; this gives $\dim Q\le d$.  Let $\bx$ be a sequence of
length $d$ in $P$, which is regular on $(R/\iR(V)\oplus S/\iS(V)\oplus
T\oplus V)$.  By Lemma \ref{regularSum}, it is also $Q$-regular, so $Q$
is Cohen-Macaulay of dimension $d$.
  \end{proof}

To describe those situations, where connected sums do not produce new
rings, we review basic properties of Hilbert-Samuel multiplicities.

  \begin{subsec}
 \label{lem:multiplicity}
Let $(P,\fp,k)$ be a Cohen-Macaulay local ring of dimension $d$.  

When $k$ is infinite, the \emph{multiplicity} $e(P)$ can be expressed as
  \[
e(P)=\inf\{\length(P/\bx P)\mid \bx\text{ is a $P$-regular sequence
in }P\}\,;
  \]
see \cite[4.7.11]{BH}.  If $P\to P'$ is a surjective homomorphism of rings,
and $P'$ is Cohen-Macaulay of dimension $d$, then by \cite[Ch.\,1,
3.3]{Sl} there exists in $P$ a sequence $\bx$ that is both $P$-regular
and $P'$-regular, and $e(P')=\length(P'/\bx P')$ holds.

When $k$ is finite, one has 
$e_P(M)=e_{P[y]_{\fp[y]}}\big(M\otimes_P P[y]_{\fp[y]}\big)$.

The ring $P$ is regular if and only if if $e(P)=1$.

It is a quadratic hypersurface if and only if $e(P)=2$.
  \end{subsec}

  \begin{proposition}
 \label{multiplicity}
Assume that the rings $R/\iR(V)$, $S/\iS(V)$, and $T$, and the $T$-module
$V$, are Cohen-Macaulay, and their dimensions are equal. 

When $R$ is regular one has $I=0$ and $R\#_TS\cong S/\iS(V)$.

When $R$ is a quadratic hypersurface and $I\ne0$, one has $R\#_TS\cong S$.
  \end{proposition}

  \begin{proof}
Set $d=\dim T$.  By Proposition \ref{cmSum}, $P$, $Q$, $R$, and $S$ are
Cohen-Macaulay of dimension $d$.  Thus, every $P$-regular 
sequence is also regular on $Q$, $R$, and $S$.

When $R$ is regular it is a domain; $\dim R=\dim T$ implies 
$I=0$, so \ref{trivial} applies.

Assume $I\ne0$ and $e(R)=2$.  Tensoring, if necessary, the diagram
\eqref{diagramSum} with $P[x]_{\fp[x]}$ over $P$, we may assume that $k$
is infinite. By \ref{lem:multiplicity}, there is a $P$- and $R$-regular
sequence $\bx$ of length $d$ in $P$, such that $\length(\ov R)=2$,
where overbars denote reduction modulo $\bx$.  {From} \eqref{eq:fiberOv1}
and $I\ne0$ one gets $\length(\ov T)=\length(\ov R)-\length(\ov I)\le1$.
This implies $\length(\ov T)=1= \length(\ov V)$, so Lemma \ref{regularSum}
and \eqref{eq:lengthSum} give $\length(\ov Q)=\length(\ov S)$.

Setting $K=\Ker(Q\to S)$, one sees that the induced sequence
  \[
0\lra \ov K\lra\ov Q\lra\ov S\lra0
  \]
is exact, due to the $S$-regularity of $\bx$, hence $\ov K=0$, and thus $K=0$.
  \end{proof}

A construction of canonical modules sets the stage for the next result.

  \begin{subsec}
    \label{dualizing}
The ideal $(0:I)$ of $R$ is a $T$-module, which is isomorphic to 
$\Hom_R(T,R)$.  Similarly, $(0:J)\cong\Hom_S(T,S)$ as $T$-modules.  If 
$R$ and $S$ are Gorenstein, $T$ is Cohen-Macaulay,  and all three rings
 have dimension $d$, then $(0:I)$ and $(0:J)$ are isomorphic $T$-modules, 
 since both are canonical modules for $T$; see \cite[3.3.7]{BH}.
    \end{subsec}

  \begin{theorem}
 \label{gorenstein}
Let $R$ and $S$ be Gorenstein local rings of dimension $d$, let $T$ be
a Cohen-Macaulay local ring of dimension $d$ and $V$ a canonical 
module for $T$.  

Let $\eR$, $\eS$, $\iR$, and $\iS$ be maps that 
satisfy the conditions in \emph{\ref{setupSum}} and, in addition,
 \[
\iR(V)=(0:I) \quad\text{and}\quad \iS(V)=(0:J) \,.
  \] 

If $I\ne0$ or $J\ne0$, then $R\#_TS$ is a Gorenstein local ring of dimension $d$.
  \end{theorem}

  \begin{Remark}
The condition $I\ne0$ is equivalent to $J\ne0$.

Indeed $I=0$ implies $R=(0:I)=\iR(V)$, hence $\eS\iS(V)=\eR\iR(V)=T$.  In
particular, for some $v\in V$ one has $\eS\iS(v)=1\in T$, hence $S=S\iS(v)
\subseteq(0:J)$, and thus $J=0$.  By symmetry, $J=0$ implies $I=0$.
  \end{Remark}

  \begin{proof}[Proof of Theorem \emph{\ref{gorenstein}}]
The $T$-module $V$ is Cohen-Macaulay of dimension $d$, 
see~\cite[3.3.13]{BH}.  The rings $R/(0:I)$ and $S/(0:J)$ have the same
property, by \cite[1.3]{PS}.  Proposition \ref{cmSum} now shows that
the ring $Q$ is Cohen-Macaulay of dimension $d$.

Choose in $P$ an $(R/\iR(V)\oplus S/\iS(V)\oplus Q\oplus T\oplus
V)$-regular sequence $\bx$ of length $d$.  It suffices to show that
$Q/\bx Q$ is Gorenstein.  The $T/\bx T$-module $V/\bx V$ is canonical,
see \cite[3.3.5]{BH}, so reduction modulo $\bx$ preserves the hypotheses
of the theorem. In view of Lemma \ref{regularSum}, we may assume that
all rings involved are artinian.

Now we have $\soc V=Tu$ for some $u\in\soc T$; see \cite[3.3.13]{BH}.
To prove that $Q$ is Gorenstein we show that $\kappa(\iota_R(u),0)$
generates $\soc Q$.  Write $q\in\soc Q$ in the form
  \[
q=\kappa(a,b)
  \quad\text{with}\quad
(a,b)\in \fr\times_T\fs=\fp\,.
  \]
As $S$ is Gorenstein, one has $\iS(u)\in\soc S\subseteq J$.  For every
$i\in I$ this gives $\eR(i)=0=\eS\iS(u)$.  Thus, $(i,\iS(u))$ is in $\fp$,
so $\kappa(i,\iS(u))\cdot q\in\fq\cdot q=0$ holds, hence
  \[
(ia,0)=(i,\iS(u))\cdot(a,b)=(\iR(x),\iS(x))
  \]
for some $x\in V$.  Since $\iS$ is injective we get $x=0$, hence $ia=0$.
As $i$ was arbitrarily chosen in $I$, this implies $a\in(0:I)$; that is,
$a=\iR(v)$ for some $v$ in $V$.  By symmetry, we conclude $b=\iS(w)$
for some $w\in V$.  As a consequence, we get
  \[
q=\kappa(\iR(v),\iS(w))
  \quad\text{with}\quad
v,w\in V\,.
  \]

Pick any $t$ in $\ft$, then choose $r$ in $\fr$ and $s$ in $\fs$ with 
$\eR(r)=t=\eS(s)$.  Thus, $(r,s)$ is in $\fp$, hence $\kappa(r,s)$ is in 
$\fq$, whence $\kappa(r,s)\cdot q=0$.  We then have
  \begin{align*}
(\iR(tv),\iS(tw))=(r\iR(v),s\iS(w))=(r,s)\cdot(\iR(v),\iS(w))=(\iR(y),\iS(y))
  \end{align*}
for some $y\in V$.  This yields $\iR(tv)=\iR(y)$ and $\iS(tw)=\iS(y)$, hence 
$tv=y=tw$, due to  the injectivity of $\iR$ and $\iS$; in other words,
$t(v-w)=0$.  Since $t$ was an arbitrary element of $\ft$, we get 
$\ft(v-w)=0$, hence $v=w+t'u$ for some $t'\in T$.  Choosing $r'$ 
in $R$ and $s'$ in $S$ with $\eR(r')=t'=\eS(s')$, we have $(r',s')\in P$ and
  \begin{align*}
q&=\kappa(\iR(w),\iS(w))+\kappa(\iR(t'u),0)
  \\
& =\kappa\big((r',s')\cdot(\iR(u),0)\big)
  \\
& =\kappa(r',s')\cdot\kappa(\iR(u),0)
  \end{align*}
As $q$ can be any element of $\soc Q$, we get $\soc Q=
Q\cdot\kappa(\iR(u),0)$, as desired. 
  \end{proof}

  \section{Examples and variations}
    \label{Examples}
    
We collect examples to illustrate the hypotheses and the conclusions of 
results proved above, and review variants and antecedents of the notion 
of connected sum.  

Seemingly minor perturbations of diagram \eqref{diagramSum} may lead to 
non-isomorphic connected sum rings.  Next we produce a concrete 
illustration.  See also Example \ref{sah} for connected sums that are not 
isomorphic \emph{as graded algebras}. 

  \begin{example}
    \label{fermat}
Over the field $\bbQ$ of rational numbers, form the algebras 
  \[
R=\bbQ[x]/(x^3)\,,
  \quad
S=\bbQ[y]/(y^3)\,,
  \quad\text{and}\quad
T=\bbQ\,.
  \]
Letting both $\eR\colon R\to T$ and $\eS\colon S\to T$ be the canonical 
surjections, one gets 
  \[
R\times_TS=\bbQ[x,y]/(x^3,xy,y^3)\,.
 \]

Set $V=\bbQ$ and let $\iR\colon V\to R$ and $\iS\colon V\to S$ be the maps
$q\mapsto qx$ and $q\mapsto qy$, respectively.  The connected sum defined 
by these data is a local ring $(Q,\fq,k)$ with
  \[
Q=\bbQ[x,y]/(x^2-y^2,xy)\,.
  \]

On the other hand, take the same maps $\eR$, $\eS$, and $\iS$ as above, 
and replace $\iR$ with the map $q\mapsto pq$, where $p$ is a prime number 
that is not congruent to $3$ modulo~$4$.  We then get as connected sum 
a local ring $(Q',\fq',\bbQ)$ with
  \[
Q'=\bbQ[x',y']/(x'^2-py'^2,x'y')\,.
  \]

We claim that these rings are not isomorphic.  In fact, more is true:  

Every ring homomorphism $\kappa\colon Q'\to Q$ satisfies $\kappa(\fq')
\subseteq \fq^2$.  

Indeed, any ring homomorphism of $\bbQ$-algebras is $\bbQ$-linear, 
so $\kappa$ is a homomorphism of $\bbQ$-algebras.  The images of
$x'$ and $y'$ can be written in the form
  \begin{align*}
\kappa(x')&=ax+by+cy^2
\\
\kappa(y')&=dx+ey+fy^2
    \end{align*}
for appropriate rational numbers $a$, $b$, $c$, $d$, $e$, and $f$.  
In $Q$ this gives equalities 
  \begin{align*}
(a^2+b^2)x^2=a^2x^2+b^2y^2=\kappa(x'^2)=\kappa(py'^2)=p(d^2x^2+e^2y^2)
=p(d^2+e^2)x^2
    \end{align*}

We need to show that the only rational solution of the equation
  \begin{equation}
    \label{eq:fermat}
a^2+b^2=p(d^2+e^2)
  \end{equation}
is the trivial one.  If not, then $a^2+b^2\ne0$.  Clearing denominators, we 
may assume $a,b,c,d\in\bbZ$ and write $a^2+b^2=p^ig$ and $d^2+e^2=p^jh$ 
with integers $g,h,i,j\ge0$, such that $gh$ is not divisible by $p$.  By 
Fermat's Theorem on sums of two squares, see \cite[\S5.6]{Sm}, $i$ and $j$ 
must be even.   This is impossible, as \eqref{eq:fermat} forces $i=j+1$.  
    \end{example}

Now we turn to graded rings and degree-preserving homomorphisms.

Recall that the \emph{Hilbert series} of a graded vector space $D$ over a
field $k$, with $\rank_kD_n<\infty$ for all $n\in\bbZ$ and $D_n=0$ for $n\ll0$, 
is the formal Laurent series
   \[
H_D=\sum_{n>-\infty}\rank_k(D_n)z^n\in\bbZ[\![z]\!][z^{-1}]\,.
  \]

  \begin{remark}
    \label{gradedProd}
Let $k$ be a field and assume that the rings $R$ and $S$ in diagram
\eqref{diagramProd} are commutative finitely generated $\bbN$-graded
$k$-algebras with $R_0=k=S_0$, and the maps are homogeneous.  Equation
\eqref{eq:lengthProd} then can be refined to: \begin{equation}
    \label{eq:hilbertProd}
H_{R\times_TS}=H_R+H_S-H_T\,,
 \end{equation}
and the obvious version of Theorem \ref{cmProd} for graded rings holds as well.  

Assume, in addition, that in the diagram \eqref{diagramSum} all maps
are homogeneous.  Equation \eqref{eq:lengthSum} then can be refined to:
\begin{equation}
    \label{eq:hilbertSum}
H_{R\#_TS}=H_R+H_S-H_T-H_V\,.
 \end{equation}
  \end{remark}

Diligence is needed to state a graded analog of Theorem \ref{gorenstein}.
Recall that for each finite graded $T$-module $N$ one has $H_M=h_N/g_N$
with $h_N\in\mathbb Z[z^{\pm 1}]$ and $g_N\in\mathbb Z[z]$, and that the
integer $a(N)=\deg(h_N)-\deg(g_N)$ is known as the \emph{$a$-invariant} of $N$.

  \begin{theorem}
    \label{gorensteinAnlog}
Let $R\xra{\eR} T\xla{\eS}S$ be surjective homomorphisms of commutative
$\bbN$-graded $k$-algebras of dimension $d$, with $R_0=k=S_0$.  Assume 
that $R$ and $S$ are Gorenstein, $T$ is Cohen-Macaulay, and $V$ is
a canonical module for $T$.

A connected sum diagram \eqref{diagramSum}, with $\iR$ and $\iS$
isomorphisms of graded modules, exists if and only if $a(R)=a(S)$.
When this is the case the graded algebra $R\#_TS$ is Gorenstein of 
dimension $d$, with $a(R\#_TS)=a(R)$ and
 \begin{equation}
    \label{eq:hilbertSumA}
H_{R\#_TS}(z)=H_R(z)+H_S(z)-H_T(z)-(-1)^dz^{a(R)}\cdot H_T(z^{-1})\,.
 \end{equation}
  \end{theorem}

  \begin{proof}
{From} \cite[4.4.5]{BH} one obtains
  \[
H_{\Hom_R(T,R)}(z)
=z^{a(R)}\cdot H_{\Hom_R(T,R(a))}(z)
=(-1)^dz^{a(R)}\cdot H_T(z^{-1})\,,
  \]
and a similar formula with $S$ in place of $R$. Thus, $\Hom_R(T,R)
\cong\Hom_S(T,S)$ holds as \emph{graded} $T$-modules if and only if
$a(R)=a(S)$.  In this case, Theorem \ref{gorenstein} (or its proof)
shows that $R\#_TS$ is Gorenstein, and formula \eqref{eq:hilbertSum}
yields \eqref{eq:hilbertSumA}.
  \end{proof}

Generation in degree $1$ does not transfer from $R$ and $S$ to 
$R\times_TS$ or $R\#_TS$:

   \begin{example} 
  \label{nonstandard}
Set $T=k[z]/(z^2)$ and form the homomorphisms of $k$-algebras
  \[
R=k[x]/(x^5)\to T\gets k[y]/(y^5)=S
  \quad\text{with}\quad
x\mapsto z \gets y\,.
  \]
Choose $V=T$ and define homomorphisms $R\xla{\iR} V\xra{\iS}S$ by 
setting $\iR(1)=x^3$ and $\iS(1)=y^3$.  The graded $k$-vector space 
$R\times_TS$ has a homogeneous basis
  \[
\{(x^i,y^i)\}_{0\le i\le4}\cup\{(0,y^j)\}_{2\le j\le4}\,,
  \]
which yields $R\times_TS\cong k[u,v]/(u^5,uv^2,v^2-u^2v)$ 
with $\deg(u)=1$ and $\deg(v)=2$. 

The canonical module of $T$ is isomorphic to $(x^3)\subset R$ and $(y^3)
\subset S$. Therefore, one gets $R \#_T S \cong k[u,v]/(v^2-u^2v,2uv-u^3)$, 
with degrees as above.
  \end{example}

  \begin{remark}
    \label{bimodules}
For the definition of connected sum given in \eqref{eq:Sum} to work in
a non-commutative context, the only change needed is to require that
the maps $\iR$ and $\iS$ in diagram \eqref{diagramSum} be homomorphisms
of $T$-\emph{bimodules}.

When the maps in the diagram are homogeneous  homomorphisms of rings
and bimodules, the resulting connected sum is a graded ring.
  \end{remark}

Remark \ref{bimodules} is used implicitly in the next two examples,
which deal with graded-commutative, rather than commutative, $k$-algebras.

  \begin{example}
    \label{manifolds}
Let $M$ and $N$ be compact connected oriented smooth manifolds of the
same dimension, say $n$.  The connected sum $M\#N$ is the manifold
obtained by removing an open $n$-disc from each manifold and gluing the
resulting manifolds with boundaries along their boundary spheres
through an  orientation-reversing homeomorphism.  The cohomology
algebras with coefficients in a field $k$ satisfy $\HH^*(M \#
N)\cong\HH^*(M)\#_k\HH^*(N)$, with $\varepsilon_{\HH^*(M)}$
and $\varepsilon_{\HH^*(N)}$ the canonical augmentations, $V=k$, and
$\iota_{\HH^*(M)}(1)$ and $\iota_{\HH^*(N)}(1)$ the orientation classes.
  \end{example}

What may be the earliest discussion of connected sums in a
ring-theoretical context followed very closely the topological model:

  \begin{example}
    \label{sah}
Sah \cite{Sa} formed connected sums of \emph{graded Poincar\'e duality algebras}
along their orientation classes, largely motivated by the following special case:

A Poincar\'e duality algebra $R$ 
with $R_i=0$ for $i\ne0,1,2$ is completely described by the quadratic form
$R_1\to k$, obtained by composing the map $x\mapsto x^2$ with the inverse 
of the orientation isomorphism $k\xra{\cong}R_2$.  Such algebras are isomorphic if 
and only if the corresponding forms are equivalent, and the connected sum of 
two such algebras corresponds to the
Witt sum of the corresponding quadratic forms.
    \end{example}

\section{Gorenstein colength}
   \label{Gorenstein colength}

Let $(Q,\fq,k)$ be an artinian local ring and $E$ an injective hull
of $k$.

The \emph{Gorenstein colength} of $Q$ is defined in \cite{A1} to be
the number
  \[
\gcl Q = \min\left\{\length(A)-\length(Q)\,\left|\,
\begin{gathered}
Q\cong A/I \text{ with $A$ an artinian}\\
\text{Gorenstein local ring}
  \end{gathered}
 \right\}\right..
  \]
One has $\gcl Q=0$ if and only if $Q$ is Gorenstein, and
  \[
0\le \gcl Q\le\length(Q)<\infty\,,
  \]
as the trivial extension $Q\ltimes E$ is Gorenstein, see \cite[3.3.6]{BH},
and $\length(Q\ltimes E)=2\length(Q)$.

  \begin{lemma}
    \label{edim}
If $Q$ is a non-Gorenstein artinian local ring and $Q\to C$ is a
surjective homomorphism with $C$ Gorenstein, then the following inequality
holds:
  \[
\gcl Q\ge\edim(Q)-(\length(Q)-\length(C))\,.
  \]
 \end{lemma}

  \begin{proof}
Let $A\to Q$ be a surjection with $(A,\fa,k)$  Gorenstein
and $\length(A)-\length(Q)=\gcl Q$.  It factors through $\ov A\to Q$, where $\ov
A=A/\soc A$.  Applying $\Hom_A(-,A)$ to the exact sequence 
$0\to\fa/\fa^2\to A/\fa^2\to A/\fa\to0$, one gets an exact sequence
  \[
0\to(0:\fa)_A\to(0:\fa^2)_A\to\Hom_A(\fa/\fa^2,A)\to0
  \]
that yields $\soc(\ov A)\cong\Hom_A(\fa/\fa^2,A)$.  As $\fa$ annihilates
$\fa/\fa^2$, the second module is isomorphic to $\Hom_k(\fa/\fa^2,k)$.
Set $K=\Ker(\ov A\to C)$.  Since $\length(\soc(C))=1$, the inclusion
$\soc(\ov A)/(K\cap(\soc(\ov A)) \subseteq\soc(C)$ gives the second
inequality below:
  \[
\length(K)\ge\length(K\cap\soc\ov A)\ge
\length(\soc(\ov A))-1=\edim A-1\ge\edim Q-1.
  \]
The desired inequality now follows from a straightforward length count:
  \[
\length(A)-\length(Q)=(\length(K)+1)-(\length(Q)-\length(C))
\ge\edim Q-(\length(Q)-\length(C))\,.\qedhere
  \]
 \end{proof}

Rings of embedding dimension $1$ need separate consideration.

  \begin{subsec}
    \label{gorbysocle}
Let $(S,\fs,k)$ be an artinian local ring with $\edim S\le1$.

The ring $S$ is Gorenstein, and one has $S\cong C/(x^n)$ with $(C,(x),k)$ a 
discrete valuation ring and $n=\length(S)$; thus, there is a surjective, but 
not bijective, homomorphism $B\to S$, where $B=C/(x^{n+1})$ is artinian,
Gorenstein, with $\length(B)=\length(S)+1$.
  \end{subsec}

  \begin{proposition}
    \label{CSandGC}
Let $(R,\fr,k)$ and $(S,\fs,k)$ be artinian local rings, with $\fr\ne0\ne\fs$.
  \begin{enumerate}[\rm(1)]
    \item
When $R$ and $S$ are Gorenstein, there is an inequality
  \[
\gcl(R \times_k S)\ge\edim R+\edim S-1\,;
  \]
equality holds if $\edim R=1=\edim S$.
    \item
When $R$ is not Gorenstein, there are inequalities
  \[
1\le\gcl(R \times_k S) \leq 
  \begin{cases}
\gcl R &\text{if}\quad\edim S=1\,;
  \\
\gcl R+\gcl S-1 &\text{if}\quad\gcl S\ge1\,.
  \end{cases}
   \]
    \end{enumerate}
  \end{proposition}

  \begin{proof} 
The ring $R\times_kS$ is not Gorenstein by Proposition \ref{cmProd},
hence $\gcl(R\times_kS)\ge1$.

(1) The ring $R\#_kS$ is Gorenstein by Theorem \ref{gorenstein},  so
apply Lemma \ref{edim} to the homomorphism $R\times_kS\to R\#_kS$ 
and use $\edim(R\times_kS)=\edim R+\edim S$.

(2)  Choose a surjective homomorphism $A\to R$ with $A$ artinian 
Gorenstein and $\length(A)=\length(R)+\gcl R$.  If $\gcl S\ge1$, 
let $B\to S$ be a surjective homomorphism with $B$ artinian Gorenstein 
and $\length(B)=\length(S)+\gcl S$; if $\edim S=1$, let $B\to S$ be the 
map described in \ref{gorbysocle}.  In both cases there is a 
commutative diagram
  \[
\xymatrixrowsep{0.8pc}
\xymatrixcolsep{1.5pc}
\xymatrix{
& A
\ar@{->>}[r]
& R 
\ar@{->>}[dr]
 \\
k
\ar@{->}[ur]
\ar@{->}[dr]
& & & k
 \\
& B
\ar@{->>}[r]
& S
\ar@{->>}[ur]
}
  \]
where two-headed arrows denote surjective homomorphisms of local rings,
and the maps from $k$ are isomorphisms onto the socles of $A$ and $B$.
Both compositions $R\gets k\to S$ are zero, so there is a surjective homomorphism 
$A\#_kB\to R\times_kS$.  

In the following string the inequality holds because $A\#_kB$
is Gorenstein, see Theorem \ref{gorenstein}, and the first equality comes
from formulas \eqref{eq:lengthProd} and \eqref{eq:lengthSum}:
  \begin{align*}
\gcl(R\times_kS)
&\le \length(A\#_kB)-\length(R\times_kS)\\
&=(\length(A)+\length(B)-2)-(\length(R)+\length(S)-1)\\
&=\gcl R+(\length(B)-\length(S)-1)
  \end{align*}
The desired upper bounds now follow from the choice of $B$.
  \end{proof}

As a first application, we give a new, simple proof of a result of Teter, 
\cite[2.2]{Te}.

 \begin{corollary}
    \label{squareZero}
A local ring $(Q,\fq,k)$ with $\fq^2 =0$ has $\gcl Q=1$ or $\edim Q\le1$.
  \end{corollary}

  \begin{proof}
The condition $\fq^2=0$ is equivalent to $\fq=\soc Q$.  Set $\rank_k\fq=s$.  

One has $s=\edim Q$, so we assume $s\ge2$; we then have $\gcl Q\ge1$.  
Lemma~\ref{lem:socle} gives $Q\cong R\times_kS$ where
$(R,\fr,k)$ and $(S,\fs,k)$ are local rings, $\fr^2=0$, 
$\edim R=s-1$, and $\edim S=1$.  If $s=2$, 
then $\edim R=1$, hence $\gcl Q=1$ by Proposition \ref{CSandGC}(1).  
If $s\ge3$, then $\gcl R=1$ holds by induction, so 
Proposition \ref{CSandGC}(2) yields $\gcl Q=1$.
   \end{proof}

Note that the conditions $\gcl Q=1$ and $\edim Q\le1$ are mutually 
exclusive; one or the other holds if and only if $R$ is isomorphic to
the quotient of some artinian Gorenstein ring by its socle, see 
\ref{gorbysocle}.  Such rings are characterized as follows:

  \begin{subsec}
    \label{teter}
Let $(Q,\fq,k)$ be an artinian local ring and $E$ an injective envelope
of $k$.

Teter \cite[2.3, 1.1]{Te} proved that there exists an isomorphism $Q\cong A/\soc(A)$,
with $(A,\fa,k)$ an artinian Gorenstein local ring, if and only if 
there is a homomorphism of $Q$-modules $\varphi\colon\fq\to\Hom_Q(\fq,E)$ 
satisfying $\varphi(x)(y)=\varphi(y)(x)$ for all $x,y\in\fq$.

His analysis includes the following observation:  $E\cong \Hom_{A}(Q,A)$, so the exact sequence 
$0\to k\to A\to Q\to0$ induces an exact sequence $0\to E\to A\to k\to0$.
It yields $E\cong\fa$, and thus a composed $Q$-linear surjection
$E\cong\fa\to\fa/\soc(A)=\fq$. 
  \end{subsec}

Using Teter's result, Huneke and Vraciu proved a partial converse:

  \begin{subsec}
    \label{HVthm}
If $\soc Q\subseteq\fq^2$, $2$ is invertible in $Q$, and there exists
an epimorphism $E\to\fq$, then $Q\cong A/\soc(A)$ with $A$ Gorenstein;
see \cite[2.5]{HV}.
  \end{subsec}

We lift the restriction on the socle of $Q$.

 \begin{theorem}
   \label{gcl1}
Let $(Q,\fq,k)$ be an artinian local ring, in which $2$ is invertible,
and let $E$ be an injective hull of~$k$.

If there is an epimorphism $E\to\fq$, then $Q\cong A/\soc(A)$ with $A$
Gorenstein.
 \end{theorem}

  \begin{proof}
By Lemma \ref{lem:socle}, there is an isomorphism $Q\cong R\times_kS$,
where $(R,\fr,k)$ is a local ring with $\soc(R)\subseteq\fr^2$ and $(S,\fs,k)$ 
is a local ring with $\fs^2=0$.  Choose a surjective homomorphism $P\to Q$ 
with $P$ Gorenstein and set $E_R=\Hom_P(R,P)$ and $E_S= \Hom_P(S,P)$.  
We then have $E\cong\Hom_P(Q,P)$ and surjective homomorphisms
  \[
E_R\oplus E_S\xra{\,\alpha\,} E
\xra{\,\beta\,}\fq=\fr\oplus\fs\xra{\,\gamma\,}\fr
  \]
where $\alpha$ is induced by the composition $Q\cong R\times_kS
\hookrightarrow R\oplus S$, $\beta$ comes from the hypothesis, and 
$\gamma$ is the canonical map.  Note that $\length(E)=\length(Q)>
\length(\fq)=\length(\beta(E))$ implies $\Ker(\beta)\ne0$; since
$\length(\soc E)=1$, we get $\soc E\subseteq\Ker(\beta)$.

One has $\fq^2\alpha(E_S)=\alpha(\fq^2E_S)=0$.  This gives
$\fq\alpha(E_S)\subseteq\soc(E)\subseteq\Ker(\beta)$, hence
$\fq\beta\alpha(E_S)=\beta(\fq\alpha(E_S))=0$, and thus 
$\beta\alpha(E_S)\subseteq\soc\fq$.  {From} here we get
  \[
\gamma\beta\alpha(E_S)\subseteq\gamma(\soc\fq)
\subseteq\soc\fr\subseteq\soc R\subseteq\fr^2\,.
  \]
Using the inclusions above, we obtain a new string: 
  \[
\fr=\gamma\beta\alpha(E_R\oplus E_S)
=\gamma\beta\alpha(E_R)+\gamma\beta\alpha(E_S)
\subseteq\gamma\beta\alpha(E_R)+\fr^2\,.
  \]
By Nakayama's Lemma, $\gamma\beta\alpha$ restricts to a surjective
homomorphism $E_R\to\fr$.

As $E_R$ is an injective envelope of $k$ over $R$, and $\soc R$ is
contained in $\fr^2$, we get $\gcl R=1$ or $\edim R\le1$ from 
Huneke and Vraciu's theorem; see \ref{HVthm}.  On the other hand, we
know from Lemma \ref{squareZero} that $S$ satisfies $\gcl S=1$
or $\edim S\le1$, so from Proposition \ref{CSandGC} we conclude that 
$\gcl Q=1$ or $\edim Q\le1$ holds.
  \end{proof}

Finally, we take a look at the values of $\length(A)-\length(Q)$, when 
$Q$ is fixed.

  \begin{remark}
Let $Q$ be an artinian local ring; set $\edim Q=e$ and $\gcl Q=g$.

If $e\le 1$ or $g\ge1$, then for every $n\ge 0$ there is an isomorphism 
$Q\cong A/I$, with $A$ a Gorenstein local ring and $\length(A)-\length(Q)=g+n$.

Indeed, the case of $e=1$ is clear from \ref{gorbysocle}, so we  assume 
$e\ge2$.  When $g\ge1$,  let $R\to Q$ be a surjective homomorphism 
with $R$ Gorenstein and $\length(R)=g$.  For $S= k[x]/(x^{n+2})$, the 
canonical surjection $R\times_kS\to R\times_kk\cong R$ maps 
$\soc(R)\oplus\soc(S)$ to zero, and so factors through $R\#_kS$.  
Theorem \ref{gorenstein} shows that this ring is Gorenstein, and 
formula \eqref{eq:lengthSum} yields $\length(R\#_kS) = g+(n+2)-2$.
 \end{remark}

\section{Cohomology algebras}
  \label{sec:Cohomology algebras}

Our next goal is to compute the cohomology algebra of a connected sum
of artinian Gorenstein rings over their common residue field, in terms
of the cohomology algebra of the original rings.  The computation takes
up three consecutive sections.

In this section we describe some functorial structures on cohomology.

  \begin{subsec}
   \label{pi}
Let $(P,\fp,k)$ be a local ring and $\kappa\colon P\to Q$ is a surjective
ring homomorphism.

Let $F$ be a minimal free resolution of $k$ over $P$.  One then has
  \[
\Ext^*_P(k,k)=\Hom_P(F,k)
  \quad\text{and}\quad
\Tor^P_{*}(k,k)=F\otimes_Pk\,.
  \]

\emph{Homological products} turns $\Tor^P_{*}(k,k)$ into a 
graded-commutative algebra with divided powers, see \cite[2.3.5]{GL}
or \cite[6.3.5]{Av:barca}; this structure is preserved by the map
  \[
\Tor^{\kappa}_*(k,k)\colon\Tor^P_*(k,k)\to\Tor^Q_*(k,k)\,.
  \]

\emph{Composition products} turn $\Ext^*_P(k,k)$ into a graded $k$-algebra
see \cite[Ch.\,II, \S3]{GL}, and the homomorphism of rings $\kappa$ induces a 
homomorphism of graded $k$-algebras
  \[
\Ext^*_{\kappa}(k,k)\colon\Ext^*_{Q}(k,k)\to\Ext^*_{P}(k,k)\,.
  \]

For each $n\in\bbZ$, the canonical bilinear pairing
  \[
\Ext^n_{P}(k,k)\times\Tor_n^P(k,k)\to k
  \]
given by evaluation is non-degenerate; we use it to identify the graded
vector spaces
  \[
\Ext^*_P(k,k)=\Hom_k(\Tor^P_*(k,k),k)\,.
  \]

Let $\pi^*(P)$ be the graded $k$-subspace of $\Ext^*_P(k,k)$, consisting 
of those elements that vanish on all products of elements in $\Tor^P_{+}(k,k)$ 
and on all divided powers $t^{(i)}$ of elements $t\in\Tor^P_{2j}(k,k)$ with
$i\ge2$ and $j\ge1$. As $\pi^*(P)$ is closed under graded commutators in
$\Ext^*_P(k,k)$, it is a graded Lie algebra, called the \emph{homotopy Lie 
algebra} of $P$.  The canonical map from the universal enveloping algebra 
of $\pi^*(P)$ to $\Ext^*_P(k,k)$ is an isomorphism; see \cite[10.2.1]{Av:barca}.   
The properties of $\Tor^{\kappa}_*(k,k)$ and $\Ext^*_{\kappa}(k,k)$ 
show that $\kappa$ induces a homomorphism of graded Lie algebras
  \[
\pi^*(\kappa)\colon\pi^*(Q)\to\pi^*(P)\,.
  \]
  
The maps $\Tor^{\kappa}_*(k,k)$, $\Ext^*_{\kappa}(k,k)$, and 
$\pi^*(\kappa)$ are functorial.
  \end{subsec}

The next lemma can be deduced from \cite[3.3]{Av:Golod}.  We provide
a direct proof.

   \begin{lemma}
     \label{iota}
Given a local ring $(P,\fp,k)$ and an exact sequence of $P$-modules
 \[
\xymatrixrowsep{2pc}
\xymatrixcolsep{2pc}
\xymatrix{
0 \ar[r]
 & V \ar[r]^{\iota}
 & P \ar[r]^{\kappa} 
 & Q \ar[r] 
 & 0 
}
 \]
there is a natural exact sequence of $k$-vector spaces
  \[
\xymatrixcolsep{2pc}
\xymatrix{
0\ar[r]
&\pi^1(Q) \ar[r]^-{\pi^1(\kappa)}
&\pi^1(P) \ar[r]
&\Hom_P(V,k)\ar[r]^-{\wt\iota}
&\pi^2(Q) \ar[r]^-{\pi^2(\kappa)}
&\pi^2(P)
}
  \]
  \end{lemma}

  \begin{proof}
The classical change of rings spectral sequence 
  \[
\mathrm{E}^{p,q}_2=\Ext^p_{Q}(k,\Ext^q_{P}(Q,k))
\underset p{\implies}\Ext^{p+q}_{P}(k,k),
  \]
see \cite[XVI.5.(2)${}_4$]{CE}, yields a natural exact sequence of terms of 
low degree
  \begin{equation}
    \label{eq:edge}
  \begin{gathered}
\xymatrixcolsep{1.8pc}
\xymatrixrowsep{0.3pc}
\xymatrix{
&{\qquad\qquad}0\ar[r]
&\Ext^1_{Q}(k,k) \ar[rr]^-{\Ext^1_{\kappa}(k,k)}
&&\Ext^1_{P}(k,k)
  \\
{\ }\ar[r]
&\Ext^1_P(Q,k) \ar[r]^-{\delta}
&\Ext^2_{Q}(k,k) \ar[rr]^-{\Ext^2_{\kappa}(k,k)}
&&\Ext^2_{P}(k,k)
}
  \end{gathered}
  \end{equation}

Next we prove $\Ima(\delta)\subseteq\pi^2(Q)$.  Indeed, $\Tor^P_2(k,k)$
contains no divided powers, so $\pi^2(P)$ is the subspace of $k$-linear
functions vanishing on $\Tor^Q_1(k,k)^2$.  Dualizing the exact sequence above, 
one obtains an exact sequence
  \[
\xymatrixcolsep{1.8pc}
\xymatrixrowsep{0.3pc}
\xymatrix{
\ar[r]
&\Tor_2^{P}(k,k) \ar[rr]^-{\Tor_2^{\kappa}(k,k)}
&&\Tor_2^{Q}(k,k) \ar[rr]^-{\Hom_k(\delta,k)}
&&\Tor_1^P(Q,k)
  \\
\ar[r]
&\Tor_1^{P}(k,k) \ar[rr]^-{\Tor_1^{\kappa}(k,k)}
&&\Tor_1^{Q}(k,k)\ar[rr]
&&0{\qquad\qquad}
}
   \]
of $k$-vector spaces.  Since $\Tor_*^{\kappa}(k,k)$
is a homomorphism of algebras, it gives
  \[
\Tor^Q_1(k,k)^2=(\Ima(\Tor_1^{\kappa}(k,k))^2
\subseteq\Ima(\Tor_2^{\kappa}(k,k))=\Ker(\Hom_k(\delta,k))\,.
  \]
Thus, for each $\epsilon\in\Ext^1_P(Q,k)$ one gets
$\delta(\epsilon)(\Tor^Q_1(k,k)^2)=0$, as desired.

The exact sequence in the hypothesis of the lemma induces an isomorphism
    \begin{equation}
    \label{eq:eth}
\eth\colon \Hom_P(V,k)\xra{\ \cong\ }\Ext^1_P(Q,k)\,.
 \end{equation}
of $k$-vector spaces.  Setting $\wt\iota=\delta\eth$, and noting that
one has $\pi^1(P)=\Ext^1_P(k,k)$ and $\pi^1(Q)=\Ext^1_Q(k,k)$, one gets
the desired exact sequence from that for $\Ext$'s.
  \end{proof}

The following definition uses \cite[4.6]{Av:Golod}; see \ref{Gol} for the 
standard definition.

  \begin{subsec}
   \label{ex:Gol}
A surjective homomorphism $\kappa\colon P\to Q$ is said to be 
\emph{Golod} if the induced map $\pi^*(\kappa)\colon\pi^*(Q)\to\pi^*(P)$
is surjective and its kernel is a free Lie algebra.  

When $\kappa$ is Golod $\Ker(\pi^*(\kappa))$ is  the free Lie algebra 
on a graded $k$-vector space $W$, with $W^i=0$ for $i\le1$ 
and $\rank_k W^i=\rank_k\Ext^{i-2}_P(Q,k)$ for all $i\ge2$.
  \end{subsec}

   \begin{proposition}
 \label{TensorAlgebra}
Let $\wt V$ denote the graded vector space with $\wt V{}^i=0$ for
$i\ne2$ and $\wt V{}^2=\Hom_P(V,k)$, let $\ta(\wt V)$ be the tensor
algebra of $\wt V$, and  let
  \[
\iota^*\colon \ta(\wt V)\to\Ext^*_Q(k,k)
  \]
be the unique homomorphism of graded $k$-algebras with $\iota^2=\wt\iota$;
see Lemma \emph{\ref{iota}}.

If $\beta$, $\gamma$, $\kappa$, and $\kappa'$ are surjective homomorphisms
of rings, the diagram
  \[
\xymatrixrowsep{2pc}
\xymatrixcolsep{2pc}
\xymatrix{
0 \ar[r]
 & V \ar[r]^{\iota} \ar[d]_-{\alpha}
 & P \ar[r]^{\kappa} \ar[d]_-{\beta}
 & Q \ar[r] \ar[d]_{\gamma} 
 & 0 
\\
0 \ar[r]
 & V' \ar[r]^{\iota'}
 & P' \ar[r]^{\kappa'}
 & Q' \ar[r]
 & 0 
}
 \]
commutes, and its rows are exact, then the following maps are equal:
  \[
\iota^*\circ\ta(\Hom_{\beta}(\alpha,k))=\Ext^*_{\gamma}(k,k)\circ\wt\iota'^*
\colon \ta(\wt{V}{}')\to \Ext^*_{Q}(k,k)\,.
    \]

If $V$ is cyclic and $\iota(V)$ is contained in $\fp^2$, or if the
homomorphism $\kappa$ is Golod, then $\iota^*$ is injective, and
$\Ext^*_Q(k,k)$ is free as a left and as a right $\ta(\wt V)$-module.
  \end{proposition}

  \begin{proof}
The maps $\wt\iota$ and $\wt\iota{}'$ are the compositions 
of the rows in the following diagram, which commutes by the 
naturality of the maps $\eth$ from \eqref{eq:eth} and $\delta$ from
\eqref{eq:edge}:
\[
\xymatrixrowsep{2.5pc}
\xymatrixcolsep{3.5pc}
\xymatrix{
{\Hom_P(V,k)}\ar[r]^-{\eth}
&\Ext^1_P(Q,k) \ar[r]^{\delta}
&\Ext^2_Q(k,k)
 \\
{\Hom_{P'}(V',k)}\ar[r]^-{\eth'} \ar[u]^{\Hom_{\beta}(\alpha,k)}
&\Ext^1_{P'}(Q',k) \ar@{->}[u]_{\Ext^1_{\pi}(\beta,k)}
\ar[r]^{\delta'}
&\Ext^2_{Q'}(k,k) \ar@{->}[u]_-{\Ext^2_{\gamma}(k,k)}
}
\]

Set $W^2=\iota^2(\wt V^2)$.  The subalgebra $E=\iota^*(\ta(\wt V))$ of
$\Ext^*_{Q}(k,k)$ is generated by $W^2$.  Lemma \ref{iota} shows that
$W^2$ is contained in $\pi^2(Q)$, so $E$ is the universal enveloping
algebra of the Lie subalgebra $\omega^*$ of $\pi^*(Q)$, generated
by $W^2$.

The Poincar\'e-Birkhoff-Witt Theorem (e.g., \cite[10.1.3.4]{Av:barca})
implies that the universal enveloping algebra $U$ of $\pi^*(Q)$ is free
as a left and as a right $E$-module.  Recall, from \ref{pi}, that $U$
equals $\Ext^*_Q(k,k)$.  Thus, it suffices to show that $\iota^*$
is injective.  This is equivalent to injectivity of $\iota^2$ plus
freeness of the associative $k$-algebra $E$; the latter condition can
be replaced by freeness of the Lie algebra $\omega^*$.

If $V$ is contained in $\fp^2$, then $\Ext^1_{\kappa}(k,k)$ is surjective,
so $\iota^2$ is injective by Lemma \ref{iota}.  If $V$ is, in addition,
cyclic, then $W^2$ is a $k$-subspace of $\pi^*(Q)$, generated by a
non-zero element of even degree.  Any such subspace is a free Lie
subalgebra.

When $\kappa$ is Golod, $\pi^1(\kappa)$ is surjective by \ref{ex:Gol},
so $\iota^2$ is injective by Lemma \ref{iota}.  Now $\Ker\pi^*(\kappa)$
is a free Lie algebra, again by \ref{ex:Gol}, hence so is its subalgebra
$\omega^*$.
  \end{proof}

\section{Cohomology of fiber products}
  \label{sec:Cohomology of fiber products}

The cohomology algebra of fiber products is known, and its structure 
is used in the next section.  To describe it, we recall a construction
of coproduct of algebras.

\begin{subsec}
   \label{coproduct}
Let $B$ and $C$ be graded $k$-algebras, with $B^0=k=C^0$ and $B^n=0=C^n$
for all $n<0$.  Thus, there exist isomorphisms $B\cong\ta(X)/K$ and
$C\cong\ta(Y)/L$, where $X$ and $Y$ are graded $k$-vector spaces, and
$K$ and $L$ are ideals in the respective tensor algebras, satisfying
$K\subseteq X\otimes_kX$ and $L\subseteq Y\otimes_kY$.  The algebra
$B\sqcup C=\ta(X\oplus Y)/(K,L)$ is a \emph{coproduct} of $B$ and $C$
in the category of graded $k$-algebras.
 \end{subsec}

Before proceeding we fix some notation.

  \begin{subsec}
    \label{convention}
When $(R,\fr,k)$ and $(S,\fs,k)$ are local rings, we let $\eR\colon R\to k$ and 
$\eS\colon S\to k$ denote the canonical surjections, and form the commutative 
diagram
 \begin{align}
\xi=&\qquad
\begin{gathered}
\xymatrixrowsep{1pc}
\xymatrixcolsep{2.5pc}
\xymatrix{
&R
\ar[dr]^{\eR}
\\
R\times_kS{\ }
\ar[dr]_{\sigma}
\ar[ur]^{\rho}
&&{\ }k{\qquad}
  \\
&S
\ar[ur]_{\eS}
}
  \end{gathered}
  \\
\intertext{of local rings.  The induced commutative diagram of graded $k$-algebras}
 &\qquad \begin{gathered}
\xymatrixrowsep{1pc}
\xymatrixcolsep{2.2pc}
\xymatrix{
&{\quad}\Ext^*_{R}(k,k)
\ar[dr]^{\Ext^*_{\rho}(k,k)}
\\
k
 \ar[ur]
\ar[dr]
&&{\quad}\Ext^*_{R\times_kS}(k,k){\quad}
 \\
&{\quad}\Ext^*_{S}(k,k)
\ar[ur]_{\Ext^*_{\sigma}(k,k)}
}
  \end{gathered}
    \\
\intertext{see \eqref{pi}, determines a homomorphism of graded $k$-algebras }
   \label{eq:xi}
\xi^*&\colon\Ext^*_{R}(k,k)\sqcup\Ext^*_{S}(k,k)\lra\Ext^*_{R\times_kS}(k,k)\,.
   \end{align}
 \end{subsec}

The following result is \cite[3.4]{Mo}; for $k$-algebras, see also 
\cite[Ch.\,3, 1.1]{PP}.

  \begin{subsec}
    \label{coproduct cohomology}
The map $\xi^*$ in \eqref{eq:xi} is an isomorphism of graded $k$-algebras.
  \end{subsec}

To describe some invariants of modules over fiber
products, we recall that the \emph{Poincar\'e series} of a finite module 
$M$ over a local ring $(Q,\fq,k)$ is defined by
  \begin{equation*}
\poinc{Q}M = \sum_i \rank_k \Ext_Q^i(M,k)\,z^i\in\bbZ[\![z]\!]\,.
  \end{equation*}

  \begin{subsec}
    \label{DrKr}
Dress and Kr\"amer \cite[Thm.\,1]{DK} proved that each finite 
$R$-module $M$ satisfies
 \[
\poinc{R\times_kS}{M}
= \poinc{R}{M}\cdot\frac{\poinc{S}{k}}{\poinc{R}{k}
+ \poinc{S}{k}
- \poinc{R}{k}\poinc{S}{k}}\,.
  \]
Formulas for Poincar\'e series  of $S$-modules are obtained by 
interchanging $R$ and $S$.
  \end{subsec}

 \begin{proposition} 
  \label{thm:FibProdGolHom}
Let $(R,\fr,k)$ and $(S,\fs,k)$ be local rings and let $\varphi\colon R \to R'$ 
and $\psi\colon S \to S'$ be surjective homomorphisms of rings.

For the induced map $\varphi\times_k\psi\colon R\times_kS
\to R'\times_kS'$ one has an equality 
  \[
\poinc{R\times_kS}{R'\times_kS'}=
\frac{\poinc{R}{R'}\poinc{S}{k}+\poinc{S}{S'}\poinc{R}{k}-\poinc{R}{k}\poinc{S}{k}}
{\poinc{R}{k} + \poinc{S}{k} - \poinc{R}{k}\poinc{S}{k}}.
  \]
    \end{proposition}

  \begin{proof}
Set $I=\Ker(\varphi)$ and $J=\Ker(\psi)$.  The first equality below 
holds because one has $\Ker(\varphi\times_k\psi)=I\oplus J$ as ideals; 
the second one comes from \ref{DrKr}:
  \begin{align*}
\poinc{R\times_kS}{R'\times_kS'}
    &=1+z\cdot(\poinc{R\times_kS}{I}+\poinc{R\times_kS}{J})
\\
    &=1+z\cdot\left(\frac{\poinc{R}{I}\poinc{S}{k}}
      {\poinc{R}{k} + \poinc{S}{k} - \poinc{R}{k}\poinc{S}{k}}+
     \frac{\poinc{S}{J}\poinc{R}{k}}
      {\poinc{R}{k} + \poinc{S}{k} - \poinc{R}{k}\poinc{S}{k}}\right)
\\
    &=1+\frac{z}{\poinc{R}{k} + \poinc{S}{k} - \poinc{R}{k}\poinc{S}{k}}\cdot 
      \left(\frac{\poinc{R}{R'}-1}{z}\cdot\poinc{S}{k}+
      \frac{\poinc{S}{S'}-1}{z}\cdot\poinc{R}{k}\right)
\\
   &=\frac{\poinc{R}{R'}\poinc{S}{k}+\poinc{S}{S'}\poinc{R}{k}-\poinc{R}{k}\poinc{S}{k}}
{\poinc{R}{k} + \poinc{S}{k} - \poinc{R}{k}\poinc{S}{k}}.
      \qedhere
  \end{align*}
  \end{proof}

We recall Levin's \cite{Le:Gol} original definition of Golod homomorphism
in terms of Poincar\'e series.  The symbol $\preccurlyeq$ stands for termwise 
inequality of power series.

  \begin{subsec}
   \label{Gol}
Every surjective ring homomorphism $R\to R'$ with $(R,\fr,k)$ local
satisfies
  \[
\poinc{R'}k\preccurlyeq \frac{\poinc{R}k}{1+z-z\poinc{R}{R'}}\,,
  \]
see, for instance, \cite[3.3.2]{Av:barca}.  Equality holds if and only if $R\to R'$ 
is \emph{Golod}.
  \end{subsec}

The following result is due to Lescot \cite[4.1]{Ls}.

  \begin{corollary} 
  \label{cor:FibProdGolHom}
If $\varphi$ and $\psi$ are Golod, then so is $\varphi\times_k\psi$.
    \end{corollary}

  \begin{proof}
When the homomorphisms $\varphi$ and $\psi$ are Golod the following
equalities hold:
  \begin{align*}
 \frac{1}{\poinc{R'\times_kS'}{k}}
    & = \frac{1}{\poinc{R'}{k}} + \frac{1}{\poinc{S'}{k}} - 1
\\
   &= \frac{(1+z-z\poinc{R}{R'})}{\poinc{R}{k}}+ \frac{(1+z-z\poinc{S}{S'})}{\poinc{S}{k}} - 1
\\
    & = \frac{(1+z-z\poinc{R}{R'} )\poinc{S}{k} + (1+z-z\poinc{S}{S'})\poinc{R}{k} - 
                  \poinc{R}{k}\poinc{S}{k}}
               {\poinc{R}{k}\poinc{S}{k}}
\\
    & = \frac{(1+z)(\poinc{R}{k}+\poinc{S}{k}-\poinc{R}{k}\poinc{S}{k}) 
   - z(\poinc{R}{k}\poinc{S}{S'} + \poinc{S}{k}\poinc{R}{R'}-\poinc{R}{k}\poinc{S}{k})}{\poinc{R}{k}\poinc{S}{k}}
\\
    & = \frac{(1+z-z\poinc{R\times_kS}{R'\times_kS'})\big(\poinc{R}{k} + \poinc{S}{k} - \poinc{R}{k}\poinc{S}{k}\big)}
                {\poinc{R}{k}\poinc{S}{k}}
\\
   &=\frac{1+z-z\poinc{R\times_kS}{R'\times_kS'}}{\poinc{R\times_kS}{k}}\,.
\end{align*}
The first and last come from \ref{DrKr}, the second from
the definition, the penultimate one from the proposition.  Stringing them
together, we see that $\varphi\times_k\psi$ is Golod.
  \end{proof}

\section{Cohomology of connected sums}
  \label{sec:Cohomology of connected sums}

We compute the cohomology algebra of a connected sum
of local rings over certain Golod homomorphisms, using amalgams of 
graded $k$-algebras.

 \begin{subsec}
   \label{amalgam}
Let $\beta\colon B\gets A\to C\,{:}\,\gamma$ be homomorphisms of
graded $k$-algebras.  

Let $B\sqcup_AC$ denote the quotient of the 
coproduct $B\sqcup C$, see \ref{coproduct}, by the two-sided ideal 
generated by the set $\{\beta(a)-\gamma(a)\mid a\in A\}$.  It 
comes equipped with canonical homomorphisms of graded $k$-algebras 
$\gamma'\colon B\to B\sqcup_AC\gets C\,{:}{\hskip1.5pt}\beta'$, 
satisfying $\gamma'\beta=\beta'\gamma$.  The universal property of
coproducts implies that $B\sqcup_AC$ is an \emph{amalgam} of $\beta$ 
and $\gamma$ in the category of graded $k$-algebras. 

If $B$ and $C$ are free as left graded $A$-modules and as right 
graded $A$-modules, then Lemaire \cite[5.1.5 and 5.1.10]{Lm} shows 
that the maps $\gamma'$ and $\beta'$ are injective, and
  \begin{equation}
    \label{eq:amalgam2}
\frac1{H_{B\sqcup_AC}}=\frac1{H_B}+\frac1{H_C}-\frac1{H_A}\,.
     \end{equation}
  \end{subsec}

  \begin{subsec}
   \label{setup}
Given a connected sum diagram \eqref{diagramSum} with local rings
$(R,\fr,k)$ and $(S,\fs,k)$, $T=k$, and canonical surjection
$\eR$ and $\eS$, set $R'=R/\iR(V)$ and $S'=S/\iS(V)$.

We refine \eqref{diagramSum} to a commutative diagram
 \begin{equation}
\Xi=\qquad  \begin{gathered}
\xymatrixrowsep{2pc}
\xymatrixcolsep{0.9pc}
\xymatrix{
&&& R
\ar@{->>}[rr]_(.3){\varphi} 
\ar@/^2.8pc/[drrrrr]^{\eR}
&& R' 
\ar@/^1pc/[drrr]
 \\
V 
\ar@/^1pc/[urrr]^-{\iR}
\ar@{->}[rr]^-{\iota}
\ar@/_1pc/[drrr]_-{\iS}
&& R\times_kS
\ar@{->>}[ur]^-{\rho}
\ar@{->>}[rr]^-{\kappa}
\ar@{->>}[dr]_-{\sigma}
&& R\#_kS
\ar@{->>}[ur]^-{\rho'}
\ar@{->>}[rr]^-{\varkappa}
\ar@{->>}[dr]_-{\sigma'}
&& R'\times_kS'
\ar@{->>}[rr]
\ar@{->>}[ul]
\ar@{->>}[dl]
&& k
 \\
&&& S
\ar@/_2.8pc/[urrrrr]_{\eS}
\ar@{->>}[rr]^(.3){\psi}
&& S'
\ar@/_1pc/[urrr]
}
  \end{gathered}
\end{equation}
where $\iota(v)=(\iR(v),\iS(v))$ and two-headed arrows denote canonical
surjections.

Proposition \ref{TensorAlgebra} now gives a commutative diagram of
graded $k$-algebras:
 \begin{equation}
   \label{eq:hopfDiag}
 \xymatrixrowsep{2.5pc}
 \xymatrixcolsep{.2pc}
  \begin{gathered}
 \xymatrix{
&&&&& \Ext^*_{R'}(k,k) 
\ar[dr]
\ar[dl]_(.7){\Ext^*_{\rho'}(k,k)}
 \\
\ta(\wt V) 
\ar@/^1.5pc/[urrrrr]^{\iR^*} 
\ar[rrrr]^-{\iota^*} 
\ar@/_1.5pc/[drrrrr]_{\iS^*} 
&&&&\Ext^*_{R\#_kS}(k,k)
&&\Ext^*_{R'\times_kS'}(k,k)
\ar[ll]_{\Ext^*_{\varkappa}(k,k)}
 \\
&&&&& \Ext^*_{S'}(k,k) 
\ar[ur]
\ar[ul]^(.7){\Ext^*_{\sigma'}(k,k)}
 }
  \end{gathered}
   \end{equation}

By \ref{amalgam}, the preceding diagram defines a homomorphism of 
graded $k$-algebras
  \begin{equation}
    \label{eq:induced}
\Xi^*\colon
\Ext^*_{R'}(k,k)\sqcup_{\ta(\wt V)}\Ext^*_{S'}(k,k)
\longrightarrow
\Ext^*_{R\#_kS}(k,k)\,.
   \end{equation}
 \end{subsec}

  \begin{theorem}
 \label{connected}
Assume that $\iR$ and $\iS$ in \emph{\ref{setup}} are injective and non-zero.

If the homomorphism $\varkappa\colon R\#_kS\to R'\times_kS'$ is Golod, in 
particular, if
  \begin{enumerate}[\quad\rm(a)]
 \item
the rings $R$ and $S$ are Gorenstein of length at least $3$, or
 \item
the homomorphisms $\varphi$ and $\psi$ are Golod,
   \end{enumerate}
then $\Xi^*$ in \eqref{eq:induced} is an isomorphism, and the canonical 
maps below are injective:
  \[
\Ext^*_{R'}(k,k)
\xra{\Ext^*_{\rho'}(k,k)}
\Ext^*_{R\#_kS}(k,k)
\xla{\Ext^*_{\sigma'}(k,k)}
\Ext^*_{S'}(k,k)\,.
  \]
  \end{theorem}

  \begin{corollary}
 \label{cor:connected}
When $\varkappa$ is Golod, for every $R'$-module $N$ one has
  \[
\poinc{R\#_kS}N=
\poinc{R'}N\cdot\frac{\poinc{S'}k}
{\poinc{R'}k+\poinc{S'}k-(1-rz^2)\cdot \poinc{R'}k \poinc{S'}k}\,,
  \]
where $r=\rank_kV$ (and thus, $r=1$ under condition \emph{(a)}).
Formulas for Poincar\'e series of $S'$-modules are obtained by 
interchanging $R'$ and $S'$.
   \end{corollary}

In preparation for the proofs, we review a few items.

  \begin{subsec}
   \label{specialGol}
When $(P,\fp,k)$ is a local ring and $\kappa\colon P\to Q$ a surjective 
homomorphism with $\fp\Ker(\kappa)=0$, the following inequality holds, 
with equality if and only if $\kappa$ is Golod:
  \begin{equation*}
\poinc{Q}k\preccurlyeq
\frac{\poinc{P}k}{1-\rank_k(\Ker(\kappa))\cdot z^2\cdot \poinc{P}{k}}\,.
  \end{equation*}

Indeed, the short exact sequence of $P$-modules $0\to\Ker(\kappa)\to P\to Q\to0$ 
yields $\poinc{P}{Q} = 1 + \rank_k(\Ker(\kappa))\cdot z\cdot\poinc{P}{k}$,
so the assertion follows from \ref{Gol}.
  \end{subsec}

The Golod property may be lost under composition or decomposition, but:

\begin{lemma} 
  \label{socSeqLemma}
Let $P\xra{\kappa}Q\xra{\varkappa}P'$ be surjective homomorphisms of rings.

When $\fp\Ker(\varkappa\kappa)=0$ holds, the map
$\varkappa\kappa$ is Golod if and only if $\varkappa$ and $\kappa$ are.
  \end{lemma} 

\begin{proof}
Set $\rank_k\Ker(\kappa) = r$ and $\rank_k\Ker(\varkappa) = r'$.  {From}
\ref{specialGol} one gets
\begin{align*}
\poinc{P'}{k} \preccurlyeq \frac{\poinc{Q}{k}}{1-r'z^2\cdot\poinc{Q}{k}} 
\preccurlyeq 
\frac{{\displaystyle\frac{\poinc{P}{k}}{1 - rz^2\cdot\poinc{P}{k}}}}
{1 - r'z^2\cdot{\displaystyle\frac{\poinc{P}{k}}{1-rz^2\cdot\poinc{P}{k}}}}
= \frac{\poinc{P}{k}}{1 - (r+r')z^2\cdot\poinc{P}{k}}\,.
\end{align*}
One has $\rank_k\Ker(\varkappa\kappa) = r+r'$, so the desired assertion 
follows from \ref{specialGol}.
  \end{proof}

\begin{subsec}
  \label{GolSocle} 
When $(Q,\fq,k)$ is an artinian Gorenstein ring with $\edim Q\ge2$,
the canonical surjection $Q\to Q/\soc Q$ is a Golod homomorphism; see 
\cite[Theorem 2]{LA}.
 \end{subsec}

\stepcounter{theorem}

  \begin{proof}[Proof of Theorem \emph{\ref{connected}}]
For $Q=R\#_kS$ and $P'=R'\times_kS'$, we have a commutative diagram,
with $\theta$ the canonical surjection, see \ref{amalgam}, and $\xi^*$ the 
bijection from \ref{coproduct cohomology}:
  \[
 \xymatrixrowsep{2pc}
 \xymatrixcolsep{4pc}
 \xymatrix{
\Ext^*_{R'}(k,k)\sqcup_{\ta(\wt V)}\Ext^*_{S'}(k,k)
\ar[r]^-{\Xi^*}
&\Ext^*_{Q}(k,k)
\\
\Ext^*_{R'}(k,k)\sqcup\Ext^*_{S'}(k,k)
\ar[r]^-{\cong}_-{\xi^*}\ar[u]^-{\theta}
&\Ext^*_{P'}(k,k)
\ar[u]_-{\Ext^*_{\varkappa}(k,k)}
 }
  \]
The map $\Ext^*_{\varkappa}(k,k)$ is surjective because $\varkappa$ is 
Golod, see \ref{ex:Gol}, so $\Xi^*$ is surjective.

Set $D=\Ext^*_{R'}(k,k)\sqcup_{\ta(\wt V)}\Ext^*_{R'}(k,k)$.  By Proposition 
\ref{TensorAlgebra}, $\iR^*$, $\iota^*$, and $\iS^*$ turn their targets into 
free graded $\ta(\wt V)$-modules, left and right, so \eqref{eq:amalgam2} gives:
  \[											
\frac1{H_D}
=\frac1{\poinc{R'}k}+\frac1{\poinc{S'}k}-\frac1{H_{\ta(\wt V)}}
=\frac1{\poinc{R'}k}+\frac1{\poinc{S'}k}-(1-rz^2)\,.
  \]
On the other hand, from \ref{specialGol} and \ref{DrKr} we obtain
    \begin{equation}
      \label{seriesQ}
\frac1{\poinc{Q}k}
  =\frac1{\poinc{P'}k}+rz^2
  =\left(\frac1{\poinc{R'}k}+\frac1{\poinc{S'}k}-1\right)+rz^2\,.
   \end{equation}
Thus, one has $H_D=\poinc{Q}k$.  This implies that the surjection $\Xi^*$ 
is an isomorphism.  

The injectivity of $\Ext^*_{\rho'}(k,k)$ and $\Ext^*_{\sigma'}(k,k)$ now results
from Proposition \ref{TensorAlgebra}.

It remains to show that condition (a) or (b) implies that $\varkappa$ is Golod.

(a)  Let $R$ and $S$ be artinian Gorenstein of length at least $3$.  
The socle of $R$ is equal to the maximal non-zero power of $\fr$, and
$\fr^2=0$ would imply $\length(R)=2$, so we have $\soc R\subseteq\fr^2$.
By symmetry, we also have $\soc S\subseteq\fs^2$.  

Set $P=R\times_kS$.  By definition, $Q$ equals $P/pP$, where $p$  is a
non-zero element in $\soc P$.  The maximal ideal $\fp$ of $P$ is equal to 
$\fr\oplus\fs$, so $\soc P=\soc R\oplus\soc S$ is in $\fr^2\oplus\fs^2$.  This 
gives the equality below, and \eqref{eq:local3.4} the first inequality:
  \[
\edim Q=\edim P\ge\edim R+\edim S -\edim k\ge2\,.
  \] 
Since the ring $Q$ is artinian Gorenstein by Theorem \ref{gorenstein}, and 
the kernel of the map $Q\to P'$ is non-zero and is in $\soc Q$, this homomorphism 
is a Golod by \ref{GolSocle}.

(b) If $\varphi$ and $\psi$ are Golod, then so is $\varphi\times_k\psi$  by 
Corollary~\ref{cor:FibProdGolHom}.  {From} the equality $\varphi\times_k\psi=
\varkappa\kappa$ and Lemma \ref{socSeqLemma}, one concludes that 
$\varkappa$ is Golod.  
  \end{proof}

  \begin{proof}[Proof of Corollary \emph{\ref{cor:connected}}]
As $\Ext_{\rho'}(k,k)$ is injective, the first equality in the string
  \[
{\poinc{Q}N}={\poinc{R'}N}\cdot\frac{\poinc{Q}k}{\poinc{R'}k}
={\poinc{R'}N}\cdot\frac{\poinc{S'}k}{\poinc{R'}k+
\poinc{S'}k-(1-rz^2)\poinc{R'}k\poinc{S'}k}
  \]
follows from a result of Levin; see \cite[1.1]{Le:large}.
The second one comes from \eqref{seriesQ}.
   \end{proof}

  \section{Indecomposable Gorenstein rings}
    \label{sec:ciSum}

In this section we approach the problem of identifying Gorenstein
rings that cannot be decomposed in a non-trivial way as a connected sum 
of Gorenstein local rings.  Specifically, we prove that complete 
intersection rings have no such decomposition over regular rings, 
except in a single, well understood special case.

Recall that, by Cohen's Structure Theorem, the $\fr$-adic completion 
$\widehat R$ of a local ring$(R,\fr,k)$ is isomorphic to $\wt R/K$, with 
$(\wt R,\wt\fr,k)$ regular local and $K\subseteq{\wt\fr}^2$.  One says 
that $R$ is \emph{complete intersection} (\emph{of codimension $c$}) 
if $K$ can be generated by a $\wt R$-regular sequence (of length $c$).  
A \emph{hypersurface} ring is a complete intersection ring of codimension 
$1$; it is \emph{quadratic} in case $K$ is generated by an element in 
${\wt\fr}^2\smallsetminus{\wt\fr}^3$.

We also need homological characterizations of complete intersection rings: 

  \begin{subsec}
 \label{ciLie}
A local ring $(R,\fr,k)$ is complete intersection if and only if 
$\pi^3(R)=0$, if and only if $\poinc Rk(z)=(1+z)^b/(1-z)^c$ 
with $b,c\in\bbZ$, see \cite[3.5.1]{GL}.  

If $R$ is complete intersection, then $\codim R=\rank_k\pi^2(R)=c$; 
see \cite[3.4.3]{GL}.
  \end{subsec}

Now we return to the setup and notation of Section \ref{sec:ConnSum}, 
which we recall:

  \begin{subsec}
 \label{setupSum2}
The rings in the diagram \eqref{diagramSum} are local: $(R,\fr,k)$, $(S,\fs,k)$
and $(T,\ft,k)$.

The maps $\eR$ and $\eS$ are surjective; set $I=\Ker(\eR)$ and $J=\Ker(\eS)$.

The maps $\iR$ and $\iS$ are injective.
  \end{subsec}

  \begin{theorem}
 \label{regular}
When $R$ and $S$ are Gorenstein of dimension $d$,  $T$ is regular of 
dimension $d$, and $\iR(V)=(0:I)$ and $\iS(V)=(0:J)$, the ring $R\#_TS$ is 
a local complete intersection if and only if one of the following 
conditions holds:
  \begin{enumerate}[\quad\rm(a)]
    \item
$R$ is a quadratic hypersurface ring and $S$ is a complete intersection ring.

In this case, $R\#_TS\cong S$.
    \item
$S$ is a quadratic hypersurface ring and $R$ is a complete intersection ring.

In this case, $R\#_TS\cong R$.
    \item
$R$ and $S$ are non-quadratic hypersurface rings.

In this case, $\codim(R\#_TS)=2$.
  \end{enumerate}
    \end{theorem}

  \begin{proof}
Let $(P,\fp,k)$ denote the local ring $R\times_TS$, see Lemma \ref{local1},
and $Q=R\#_TS$.

If $e(R)=1$ or $e(S)=1$, then $R\#_TS=0$ by Proposition \ref{multiplicity} and 
Theorem \ref{gorenstein}.   Else, the ring $Q$ is local, see \ref{trivial};
let $\fq$ denote its maximal ideal.

If $e(R)=2$, then $Q\cong S$ by Proposition 
\ref{multiplicity}, so $Q$ and $S$ are complete intersection simultaneously.  
The case $e(S)=2$ is similar, so we assume $e(R)\ge3$ and $e(S)\ge3$.

The $P$-modules $P$, $Q$, $R$, $S$, and $T$ are Cohen-Macaulay
of dimension $d$; see Proposition \ref{cmProd} and Theorem \ref{gorenstein}.  
Tensoring the diagram \eqref{diagramSum} with $P[y]_{\fp[y]}$ over $P$, 
we may assume that $k$ is infinite. Choose a sequence $\bx$ in $P$ that is 
regular on $P$ and $T$ and satisfies $\length(T/\bx T)=e(T)$; see \ref{lem:multiplicity}.
Since $T$ is a regular ring, we have $e(T)=1$, hence $T/\bx T=k$, so the image of $\bx$ in $T$ 
is a minimal set of generators of $\ft$.  The surjective homomorphism $Q\to T$ 
induces a surjective $k$-linear map $\fq/\fq^2\to\ft/\ft^2$, so the image of $\bx$ 
in $Q$ extends to a minimal generating set of $\fq$. 

Since $\bx$ is a system of parameters for $P$, and $Q$, $R$, and $S$ are
$d$-dimensional Cohen-Macaulay $P$-modules, $\bx$ is also a system of 
parameters for each one of them.  Thus, $\bx$ is a regular sequence 
on $Q$, $R$, and $S$.  Since $\bx$ is part of a minimal set of generators 
of $\fq$, the ring $Q$ is complete intersection of codimension $c$ if and 
only if so is $Q/\bx Q$.  Also, $R$ and $S$ are Gorenstein if and only so 
are $R/\bx R$ and $S/\bx S$, and they satisfy $\length(R)\ge e(R)\ge3$ 
and $\length(S)\ge e(S)\ge3$; see \ref{lem:multiplicity}.  Lemma \ref{regularSum} 
gives an isomorphism of rings $Q/\bx Q\cong (R/\bx R)\#_k(S/\bx S)$.
Thus, after changing notation, for the rest of the proof we may assume 
$Q=R\#_kS$, where $R$ and $S$ are artinian Gorenstein rings that are 
not quadratic 
hypersurfaces.

Let $Q$ be complete intersection and assume $\edim R\ge 2$.  
Set $R'=R/\soc R$.  Theorem \ref{connected} shows that the homomorphism
$Q\to R'$ induces an injective homomorphism of cohomology
algebras, and hence one of homotopy Lie algebras; see \ref{pi}.  This
gives the second inequality in the following string, where the first
inequality comes from \ref{ex:Gol} (because $R\to R'$ is Golod by
\ref{GolSocle}), and the equality from \ref{ciLie}:
  \[
\rank_k\Ext^1_R(R',k)\le\rank_k\pi^3(R')\le\rank_k\pi^3(Q)=0\,.
  \]
It follows that $R'$ is free as an $R$-module.  On the other hand, it is 
annihilated by $\soc R$, and this ideal is non-zero because $R$ is artinian.  
This contradiction implies $\edim R=1$, so $R$ is a
hypersurface ring.  By symmetry, so is $S$.

Conversely, if $R$ and $S$ are hypersurface rings, then Corollary 
\ref{cor:connected} gives
  \[
\poinc{Q}k
=\frac1{1-z}\cdot\frac{\displaystyle\frac1{1-z}}
{\displaystyle\frac1{1-z}+\frac1{1-z}-(1-z^2)\cdot\frac1{1-z}\cdot\frac1{1-z}}
=\frac1{(1-z)^2}\,.
  \]
This implies that $Q$ is a complete intersection ring of codimension $2$; 
see \ref{ciLie}.
  \end{proof}

  \section*{Acknowledgement}
We thank Craig Huneke for several useful discussions.

\end{document}